\documentclass[10pt]{amsart}
\usepackage{graphicx,amssymb,amsfonts,amsmath,amsthm,newlfont}
\usepackage{epsfig,url}
\usepackage{color}

\usepackage[all,2cell]{xy} \UseAllTwocells \SilentMatrices

\newtheorem{thm}{Theorem}[section]
\newtheorem{cor}[thm]{Corollary}
\newtheorem{lem}[thm]{Lemma}
\newtheorem{prop}[thm]{Proposition}
\theoremstyle{definition}
\newtheorem{defn}[thm]{Definition}
\newtheorem{propdefn}[thm]{Proposition-Definition}

\newtheorem{example}[thm]{Example}

\theoremstyle{remark}
\newtheorem{rem}[thm]{Remark}

\numberwithin{equation}{section}

\newcommand{\Z}{\mathbb Z}
\newcommand{\C}{\mathbb C}

\newcommand{\R}{\mathbb R}
\newcommand{\N}{\mathbb N}
\newcommand{\Pro}{\mathbb P}

\newcommand{\eqv}{\mathrm{eqv}}

\newcommand{\veco}{\overset{\rightarrow}{1}}

\newcommand{\Lef}{\mathbb{L} }

\newcommand{\GE}{\mathbb{G} }

\newcommand{\gr}{\mathrm{gr}}

\font \rus= wncyr10
\newcommand{\sha}{\, \hbox{\rus x} \,}

\newcommand{\MT}{\mathcal{MT}}

\newcommand{\ZZ}{\mathcal{Z}}

\newcommand{\uu}{\mathfrak{u}}
\newcommand{\rel}{\mathrm{rel}}
\newcommand{\GG}{\mathcal{G} }
\newcommand{\Eq}{\mathcal{E}^{\times}_{\partial/\partial q}}

\newcommand{\ssl}{\mathfrak{sl}}

\newcommand{\lw}{\mathrm{lw}}

\newcommand{\av}{\mathsf{a}}
\newcommand{\bv}{\mathsf{b}}

\newcommand{\CC}{\mathcal{C}}
\newcommand{\MI}{\mathcal{MI}}

\newcommand{\dV}{\check{V}}

\newcommand{\Q}{\mathbb Q}
\newcommand{\Li}{\mathrm{Li}}
\newcommand{\Lo}{\mathcal{L}}
\newcommand{\U}{\mathcal{U}}

\newcommand{\To}{\longrightarrow}

\newcommand{\G}{\mathbb{G}}

\newcommand{\x}{\mathsf{x}}

\newcommand{\Xv}{\mathsf{X}}
\newcommand{\Yv}{\mathsf{Y}}

\newcommand{\tone}{\overset{\rightarrow}{1}\!}

\newcommand{\SL}{\mathrm{SL}}

\newcommand{\Or}{\mathcal{O}}

\newcommand{\hh}{\mathrm{h}}

\newcommand{\V}{\mathcal{V}}

\newcommand{\ls}{\mathfrak{ls}}

\newcommand{\mm}{\mathfrak{m} }

\newcommand{\HH}{\mathfrak{H} }

\newcommand{\id}{\mathrm{id} }

\newcommand{\e}{\mathbf{e}}
\newcommand{\EE}{\mathbf{E}}

\newcommand{\Lie}{\mathrm{Lie}}

\newcommand{\ue}{\mathfrak{u}^{\mathrm{geom}}}
\newcommand{\Ue}{{\U}^{\mathrm{geom}}}

\newcommand{\LL}{\mathbb{L}}

\newcommand{\CCu}{\mathcal{G}}

\newcommand{\sv}{\mathrm{sv}}

\newcommand{\comp}{\mathrm{comp}}

\newcommand{\pp}{\pmb{\psi}}
\newcommand{\pls}{\mathfrak{pls}}

\begin{document}
\author{Francis Brown}
\begin{title}[A class of non-holomorphic modular forms II]{A class of non-holomorphic modular forms II :  equivariant  iterated Eisenstein integrals}\end{title}
\maketitle
\begin{abstract} We  introduce  a  new family of real analytic modular forms  on the upper half plane. They are  arguably the simplest class of `mixed' versions of modular forms  of level one and are  constructed out of real and imaginary parts of  iterated integrals of holomorphic Eisenstein series. They  form  an algebra of  functions satisfying many properties analogous to classical holomorphic modular forms. In particular, they admit expansions in $q, \overline{q}$ and $\log |q|$  involving only rational numbers and single-valued multiple zeta values.  The first non-trivial functions in this class are real analytic Eisenstein series.
\end{abstract}

\section{Introduction} 
Let 
$\HH = \{ z: \mathrm{Im} \, z >0\} $ be the upper half plane with its action of $\SL_2(\Z)$:
\begin{equation}\label{gammaact} \gamma  z  = {az + b \over cz +d} \quad \hbox{ where } \quad \gamma = \begin{pmatrix} a & b\\c & d 
  \end{pmatrix}  \in \SL_2(\Z)\ . 
  \end{equation} 
We shall say that a real analytic function 
$$f: \HH \To \C$$
is \emph{modular of weights} $(r,s)$ with $r,s\in \Z$ if for all $\gamma \in \SL_2(\Z)$ it satisfies  
\begin{equation} \label{fbimod}  f(\gamma z ) =   (c z+d)^r (c\overline{z}+d)^s f(z)\ .
\end{equation}

In this paper we construct a ring $\MI^E$  of real-analytic modular functions  on $\HH$ which are modular analogues of the single-valued polylogarithms. They are obtained by taking real and imaginary parts of iterated primitives of  holomorphic Eisenstein series, which are 
defined for all even $k\geq 4$ by
\begin{equation} \label{introEisdefn} \GE_{k}(q) = - {B_{k} \over 2k} + \sum_{n \geq 1} \sigma_{k-1}(n) q^n  \ , 
\end{equation}
where $\sigma$ denotes the divisor function. 
 The  modular analogues of the single-valued logarithm $\log |z|^2$ are  real-analytic Eisenstein series and are well-known.  The  modular analogues of  the Bloch-Wigner dilogarithm
$D(z)= \mathrm{Im} \left( \Li_2(z) + \log|z| \, \log(1-z) \right) $ already lead to completely new functions with   interesting properties  \cite{ZagFest}.  In this paper we study the entire space of such functions.   Most of our results are summarized here.

\begin{thm} \label{thmintromain} The space of real analytic functions $\MI^E$ has the following properties:

\begin{enumerate}
\item  (Expansions). Every  $ f\in \MI^E$ is modular of weights $(r,s)$ for $r,s\geq 0$. It admits a unique expansion, for some $N\geq 0$,  of the form:
\begin{equation} \label{intro: fqcoeffs}
f (q)= \sum_{k=-N}^N \LL^k \, \Big(\! \sum_{m,n\geq 0} a^{(k)}_{m,n}   q^m \overline{q}^n\Big) 
\end{equation} 
where $q= e^{2 \pi i z}$,   $\LL : =   \log |q | =  - 2 \pi  \mathrm{Im}(z)$, and  the coefficients $ a^{(k)}_{m,n}$ are single-valued multiple zeta values.

\vspace{0.05in} 
\item  (Filtrations). The space $\MI^E$ is an algebra over $\Q$,   bigraded by the modular weights $(r,s)$. Furthermore,  it admits a filtration by \emph{length} 
$$ 0 \subset \MI^E_{0} \subset \MI^E_{1} \subset \ldots \subset \MI^E_{k}\subset \ldots $$
and also by \emph{motivic weight} $M$, which is conjecturally a grading\footnote{In this paper we have chosen  to halve the $M$-weight compared to \cite{ZagFest}, in keeping with the standard convention for the weight of multiple zeta values, which is one half of the Hodge-theoretic weight.}. 
The coefficients in the expansion \eqref{intro: fqcoeffs} have $M$-weight bounded above by the $M$-weight of $f$, where $\LL$ has $M$-weight $1$.

\vspace{0.05in} 
\item (Finiteness). The subspaces  $ M_w \MI^E_{r,s}$ of motivic weight $\leq w$, and  fixed modular degree $(r,s)$, are finite-dimensional $\Q$-vector spaces.
A    function in such a subspace is uniquely determined by  a finite number of   coefficients $a^{(k)}_{m,n}$.

\vspace{0.05in} 
 \item  (Differential structure). Let $\partial$ be the differential operator of modular bidegree  $(1,-1)$ which acts on functions of modular weights $(r, \bullet)$ by 
 $$\partial_r =  (z-\overline{z}) \frac{\partial}{\partial z} + r\ .  $$ Let  $\overline{\partial}$ denote its complex conjugate of modular bidegree $(-1,1)$.  These operators are close variants of the Maass raising and lowering operators.  Then 
 \begin{align} \label{intro: partialMIEk}
 \partial \MI^E_{\ell}\  \subset \   \MI^E_{\ell} \  \oplus \    \left( E[\LL] \times \MI^E_{{\ell}-1}  \right)  \  ,  \\
 \overline{\partial} \MI^E_{\ell}\  \subset \   \MI^E_{\ell}    \ \oplus  \   \left( \overline{E}[\LL] \times \MI^E_{{\ell}-1} \right)  \  ,   \nonumber 
 \end{align} 
where $E$ is the  $\Q$-vector space generated by holomorphic Eisenstein series \eqref{introEisdefn}.  In particular, for $f \in \MI^E_{\ell}$ of length ${\ell}$ and of modular weights $(n,0)$, we have 
$$\partial f   \quad \in \quad E[\LL] \times \MI^E_{{\ell}-1}\ .$$
Therefore the space $\MI^E$ is generated by  iterated primitives of Eisenstein series with respect to $\partial, \overline{\partial}$. 
The equations (\ref{intro: partialMIEk}) also  imply  eigenvector relations  on $\MI^E$ with respect to the bigraded Laplace-Beltrami operator.  
 
\vspace{0.05in} 
 \item (Orthogonality). For every $f\in\MI^E$, the function $\partial f$ is  orthogonal to the space of holomorphic cusp forms with respect to the Petersson inner product. 
 \vspace{0.05in} 
 \item  (Algebraic structure).   Let $\lw(\MI^E) \subset \MI^E$ denote the subspace of functions with  modular weights $(n,0)$, i.e., satisfying the classical modular transformation property with no anti-holomorphic automorphy factor\footnote{the notation $\lw$ stands for `lowest weight' in the sense of $\mathfrak{sl}_2$-representations.}. Then  $\lw(\MI^E)$ is dual to the lowest weight vectors in a certain well-known Lie algebra $\ue$ of geometric derivations on the free Lie algebra on two generators. 
\end{enumerate}

\end{thm}

This theorem is not exhaustive: the  class of functions $\MI^E$  have several other properties which are proved  in this   paper, and many more which are not.
For example,  to every $f \in \MI^E$ we can associate an $L$-function which is a linear combination of Dirichlet series. It  admits   a meromorphic continuation to $\C$  and satisfies  a functional equation.
We  show that this $L$-function  is a single-variable restriction of  the multiple-variable $L$-functions recently defined in \cite{MultipleL}. 

\subsection{Examples}
\subsubsection{Length zero} 
In length zero, we have 
$$\MI^E_0 = \mathcal{Z}^{\sv}$$
where $\mathcal{Z}^{\sv}$ is the weight-filtered $\Q$-algebra of single-valued multiple zeta values, viewed as constant functions of modular weights $(0,0)$.  A definition is recalled below.  This space is already infinite-dimensional, but is finite in each $M$-weight. It does not contain any powers of $\pi$ but contains the odd zeta values 
$$\zeta^{\sv}(2n+1) = 2 \, \zeta(2n+1)$$
in $M$-weight $2n+1$, for $n\geq 1$.  The first interesting generator occurs in weight $11$:
$$\zeta^{\sv}(3,5,3) = 2 \, \zeta(3, 5, 3) -  2\, \zeta(3)\zeta(3, 5)  -10\, \zeta (3)^2\zeta(5)\ .$$
The space $\MI^E$ is  an algebra over $\mathcal{Z}^{\sv}$ and satisfies $$M_0 \MI^E_0 = M_0 \mathcal{Z}^{\sv}=\Q \ .$$

\subsubsection{Length one}  In length one, $\MI^E$ contains the first  interesting functions:
$$\MI^E_1  \cong \bigoplus_{r,s} \mathcal{E}_{r,s}  \,\mathcal{Z}^{\sv} \  $$
where the direct sum is over all integers $r,s\geq 0$  such that $w=r+s>0$ is even and
\begin{equation} \label{Eijelementary} 
\mathcal{E}_{r,s} (z)=  { w! \over   ( 2   \pi i )^{w+2}}  {1 \over 2}   \sum_{(m,n) \in \Z^2 \backslash (0,0)}   { \LL  \over (mz+n)^{r+1} (m\overline{z}+n)^{s+1} } 
\end{equation} 
are real analytic Eisenstein series, which  are  modular of weights $(r,s)$. The functions $\mathcal{E}_{r,s}$ admit an expansion of the  form:
$$\mathcal{E}_{r,s} = a_{w} \,\LL  + b_{r,s}\,  \zeta^{\sv}(w+1) \, \LL^{-w}  \  + \  \left(\hbox{a series in }   \Q[[ q, \overline{q}]][\LL^{-1}] \right)$$
where   $a_w$, $b_{r,s}$ are explicit rational numbers.   The elements $\mathcal{E}_{r,s}$ and $\LL$ have $M$-filtration equal to $1$.
 One checks that the $M$-filtration of every term in the previous expansion is  indeed $\leq 1$. 
The differential structure  amounts to the equations 
$$ \partial \, \mathcal{E}_{w,0}   =  \LL \GE_{w+2}   \quad  \hbox{ and }  \quad  \partial \,\mathcal{E}_{r,s} = (r+1) \mathcal{E}_{r+1, s-1} $$    for all $1\leq s\leq w$, and their complex conjugates.  The family of functions $\mathcal{E}_{r,s}$ are associated to a universal mixed elliptic motive whose fibre is a  simple extension of mixed Tate motives: namely,   $\Q(-w-1)$ by $\Q$, whose (single-valued) periods are rational multiples of $1$ and   $ \zeta^{\sv}(w+1)$.
 The  completed $L$-function  $\Lambda(\mathcal{E}_{a,b},s)$  associated to $\mathcal{E}_{a,b}$ is proportional to a  certain rational function of $s$ multiplied by  a product of two Riemann zeta functions $\zeta(s) \zeta(s-2w+1)$.
 
\subsubsection{Higher length}
Starting from length two,  $\MI^E$ contains  polynomials in the $\mathcal{E}_{r,s}$ together with a large  array of completely new functions, some examples of which were described  in \cite{ZagFest}. 
In length two, the enumeration is  as follows. To every pair of Eisenstein series $\G_{2m+2}$, $\G_{2n+2}$ we can take real and imaginary parts of their double iterated integral  in a prescribed way to obtain a priori $(m+1)(n+1)$  functions.   However, the condition of  orthogonality  (5) imposes exactly  one constraint (in every  relevant modular bidegree) for every holomorphic cusp form and pushes the dimension of $\MI^E_{2}$ down by that amount. The precise condition involves the critical values of $L$-functions of cusp forms, and is related to the existence of non-trivial  extensions of $\Q$ by motives of cusp forms (which are governed by non-critical values of the said $L$-functions). Thus the space of functions $\MI^E$ grows   approximately quadratically in length two and cubically in length three.  The precise  enumeration  in lengths two and three is stated in \S\ref{sect: PLS}, but is not known explicitly  in length four or above, and is closely related to deep questions about the structure of the category of mixed elliptic motives, the Broadhurst-Kreimer conjecture on depth-graded multiple zeta values, and the existence of `higher' Rankin-Selberg convolutions.

\subsection{Comments}
\begin{itemize}
\item The  class of functions  $\MI^E$  share many properties with classical holomorphic modular forms. Indeed, a key feature of the latter is  the fact that a holomorphic modular form is uniquely determined by an explicit number of Fourier coefficients depending on its modular weight. 
The same property is true for the subspace of functions in $\MI^E$ of bounded motivic weight.
It would be very interesting to  make this effective and determine an exact bound for  the number of  expansion coefficients which  determine an element of $\MI^E$ uniquely. 

\vspace{0.05in}
\item 
It should be relatively straightforward to generalise this theory to higher levels.   In a different direction, $\MI^E$ is a subspace  of a larger class of functions  which also includes iterated primitives of  cusp forms as well as Eisenstein series \cite{MMV}.  For this, one must  allows  poles at the cusps, for example.  The length one part of this  class  is  the theory of 
weak harmonic Maass forms and mock modular forms  of level one   \cite{CNHMF3}, but involves new functions thereafter. 

\vspace{0.05in}

\item
The algebra $\ue$ has occured in various guises,  most notably as the fundamental  representation of the Tannaka Lie algebra of Hain and Matsumoto's universal mixed elliptic motives. 
The   space $\MI^E$ can be viewed as an elementary realisation  of this category as a space of functions: for example the structure of $\MI^E$ encodes information about extension data in that category. 

\vspace{0.05in}

\item  The functions $\MI^E$ are expected to contain the modular graph functions studied in string theory (see below), and in fact general correlation functions for configurations of points on  the universal elliptic curve. 
If true, the theorem \ref{thmintromain} implies most of the open conjectures about this class of functions.

\vspace{0.05in}
\item The Lie algebra $\ue$ is closely related to combinatorial relations satisfied by multiple zeta values. In fact, in an earlier work we showed that it satisfies linearised versions of the double-shuffle equations, about which several important conjectures are still open. Items $(5)$ and $(6)$ taken together therefore imply that these equations are, in a precise sense, orthogonal to cusp forms. This should shed light on their structure.

\vspace{0.05in}
 \end{itemize}

\subsection{Construction}
The functions in $\MI^E$ are  the equivariant sections of the unipotent fundamental group of the universal elliptic curve. To explain this, let us recall the construction of single-valued polylogarithms in genus zero \cite{SVMP}, before passing to its elliptic generalisation. 
Consider the following multi-valued function $L(z)$ on $\C^{\times} \backslash \{1\}$ taking values in formal power series $\C\langle\langle \x_0, \x_1\rangle \rangle$ in two non-commuting variables $\x_0, \x_1$. It  is the unique solution to the Knizhnik-Zamolodochikov equation 
$$d L(z) = \omega_{KZ}\, L(z) \qquad \hbox{ where }   \qquad  \omega_{K\!Z} = {dz \over z}  \x_0  + {dz \over 1-z} \x_1 $$
with the initial condition $L(\tone_0)=1$ (denoting the regularised limit as $z\rightarrow 0$ along the real axis with unit speed),
  where $\x_0, \x_1$ act by left multiplication.  It has  a canonical  single-valued version 
  $ \Lo : \C^{\times} \backslash \{1\} \To \C\langle \langle \x_0, \x_1 \rangle \rangle $
 which satisfies  the identical differential equation and initial condition.
It can be viewed as a formal power series
$$\Lo = \sum_{w}  \Lo_w(z)  w \quad \in \quad  \C\langle \langle \x_0, \x_1 \rangle\rangle  $$
where the sum ranges over all words $w$ in $\x_0,\x_1$. Its coefficients
$\Lo_w(z)$  are  single-valued analogues of  multiple polylogarithms. Their regularised  limits  at the point $1$ are 
$$\zeta^{\sv} (w): = \Lo_w( -\tone_{1}) $$
and called single-valued multiple zeta values. They  generate $\mathcal{Z}^{\sv}$, which is strictly contained in the algebra of all multiple zeta values.  Perhaps surprisingly, the numbers $\zeta^{\sv}(w)$ satisfy all motivic relations between ordinary multiple zeta values.  

  \subsubsection{Genus 1}  The  universal elliptic KZB equation \cite{CEE, LR, HaKZB, BrLevin} that we shall use is
     $$ d J(z) = \omega J(z)$$
  where  $\omega$ is the formal one-form:
 $$\omega = -   {dq \over q}   \mathrm{ad}(\varepsilon_0)  \ +  \ \sum_{n \geq 1}  {2 \over (2n)!} \GE_{2n+2}(q) {dq \over q}  \varepsilon_{2n+2}  \ .
$$
It is the restriction of the full KZB connection to the zero section of the universal elliptic curve. 
The coefficients $\varepsilon_{2n}$, for all $n \geq 0$ are derivations of the free Lie algebra on two generators $ \Lie(\av, \bv)$. They were first written down by Tsunogai \cite{Tsunogai} in the $\ell$-adic context,  and studied by Nakamura \cite{Nakamura}.
  They are uniquely determined by the formulae $\varepsilon_{2n}[a,b]=0$,  $\varepsilon_{2n}(\av) \in  [ \Lie(\av, \bv), \Lie(\av, \bv)]$ and  
   $$\varepsilon_{2n}( \bv) =  - \mathrm{ad}(\bv)^{2n} \av \ .$$
They generate a Lie algebra we denote by $\ue$, but satisfy many interesting quadratic relations associated to cusp forms \cite{Pollack} (which will be a consequence  of the orthogonality condition (5) in theorem \ref{thmintromain}).
  Let $\Ue$ denote the affine group scheme whose Lie algebra is the completion of  $\ue$.  It admits a right action by $\SL_2$.  The function $J$ is a multivalued function of $z\in \HH$ taking values in 
  $\Ue(\C)$. 
Our first theorem proves the existence of a   genus one analogue of $\Lo$:

\begin{thm} There exists a  real-analytic function  
$$J^{\eqv} :  \HH \To \Ue(\C)$$
which satisfies the equation 
$$  { \partial \over \partial \tau}J^{\eqv}  = \omega J^{\eqv}  $$
and is equivariant for the action of $\SL_2(\Z)$, i.e., 
$$J^{\eqv}(\gamma \tau) \big|_{\gamma} = J^{\eqv}(\tau) \qquad \hbox{ for all } \gamma \in \SL_2(\Z) \  . $$
\end{thm}
The generating  series  $J^{\eqv}$, unlike its genus zero counterpart $\Lo$, is only well-defined up to right-multiplication by  an element $a\in \Ue(\mathcal{Z}^{\sv})^{\SL_2}$.  The series $J^{\eqv}$ can be viewed as a formal power series in certain \emph{coefficients} which are maps 
  $$c(J^{\eqv}): \HH \To V_{2n}\otimes_{\Q} \C$$
where    $V_{2n} = \bigoplus_{r+s=2n} X^r Y^s \Q$ is the  space of homogeneous polynomials in $X, Y$ (corresponding to $\bv$ and $\av$), 
  equipped with a right action of $\SL_2$. 
 The coefficients $c(J^{\eqv})$ are vector-valued real analytic modular forms  and satisfy:
  $$c(J^{\eqv})(\gamma \tau)  = c(J^{\eqv})(\tau) \big|_{\gamma}$$
 for all $\gamma \in \SL_2(\Z)$.  Any such function can be uniquely written  in the form
$$c(J^{\eqv})(\tau) =  \sum_{r+s=2n}  c_{r,s} (\tau)  (X-  \tau Y)^r (X- \overline{\tau} Y)^s\ , $$
where $c_{r,s} : \HH \rightarrow \C$ are real analytic and modular of weights $(r,s)$.

\begin{defn} The space $\MI^E$ is the  $\mathcal{Z}^{\sv}$-module generated by all $c_{r,s}(J^{\eqv})$. \end{defn} 
One could equivalently  set up the theory using vector-valued modular forms $c(J^{\eqv})$ instead of the constituent functions $c_{r,s}$.  This is merely a matter of choice: we opted for the latter  because of potential applications to physics.

  \subsection{Further work}
  There are a number of questions which we have not addressed. 
  \begin{enumerate}
\item   A mixed elliptic motive is a mixed Tate motive over $\Z$, equipped with an action of a certain Lie algebra on, say,  its de Rham realisation. This Lie algebra  is conjecturally
isomorphic to $\ue$. One can define motivic lifts of the functions $f \in \MI^E$ whose coefficients will be single-valued motivic multiple zeta values. 
In this precise sense, the functions in $\MI^E$ are associated to mixed Tate motives over $\Z$  carrying some extra structure. In particular, they will inherit an action 
of the motivic Galois group of mixed Tate motives over $\Z$, whose action on $\ue$ was studied in \cite{Sigma}, and is known to lowest orders. 

\item The functions in $\MI^E$  provide kernels for the Rankin-Selberg method. 
In lengths zero and one, these are known by the unfolding technique. It would be interesting to extend this to  higher lengths.
\item  The functions defined in this paper are constructed out of the action of the real Frobenius $F_{\infty}$ on the unipotent fundamental group of the universal elliptic curve.  It would be interesting to study $p$-adic versions, which should be constructed in a similar manner using the Frobenius at a finite prime. 
\end{enumerate}
The first remark  extends the similarities between $\MI^E$ and the class of classical holomorphic modular forms: whereas the latter have algebraic coefficients and admit an action of the usual Galois group, the former have coefficients which are periods, and admit an action by a motivic Galois group.

One motivation for this construction was the long standing problem of defining a natural class of functions which contain the modular graph functions arising in genus one closed string perturbation theory \cite{Graph1, Graph2, Graph3, Graph5, Zerbini}. This problem was briefly discussed in \cite{ZagFest}, and has  an extensive literature.  A  relation between modular graph functions and the class $\MI^E [ \LL^{\pm}]$ should follow by a version of the argument given in the author's thesis. The key point is that the bar  de Rham complex defined in \cite{BrLevin} has trivial cohomology.  The properties of $\MI^E$ described above should then explain several phenomena which have been observed for modular graph functions.

\subsection{Acknowledgements} This work was partly supported by ERC grant 724638 and written during a stay at the IHES.    Many thanks to Richard Hain for  ongoing discussions about relative completion and mixed elliptic motives  and to Nils Matthes, for corrections.  Many thanks also to Andrey Levin, who  suggested several years ago that double elliptic polylogarithms  could be orthogonal to cusp forms.

\section{Notations and background on the class of functions $\mathcal{M}$}
We recall some notations and background from  \cite{ZagFest}.

\subsection{General definitions} We write $z= x+iy$,  $q= \exp(2  \pi i z)$ and set
\begin{equation} \label{LLdef} \LL = \log |q| = {1 \over 2 } \log q \overline{q} =  \pi i  (z- \overline{z}) = - 2 \pi y   \ .
\end{equation}  
Consider the  bigraded algebra  of real-analytic functions \cite{ZagFest} $$\mathcal{M} = \bigoplus_{r,s} \mathcal{M}_{r,s}$$ 
 which are modular of weights $(r,s)$ and admit an expansion of the form \eqref{intro: fqcoeffs}
 where $a^{(k)}_{m,n}$ are arbitrary complex numbers. The quantity $w=r+s$  is called  the total modular weight. The space  $\mathcal{M}_{r,s}=0$ if $w$ is odd.
 The  constant part  of $f$  (called the Laurent polynomial or zeroth Fourier mode in the physics literature)   is defined to be
$$f^0   =  \sum_k  a^{(k)}_{0,0}   \LL^k   \qquad \in \qquad  \C [\LL^{\pm}] \ .$$
The subspace of cusp forms $\mathcal{S} \subset \mathcal{M}$ is defined to be $\ker (f \mapsto f^0 : \mathcal{M} \rightarrow \C[\LL^{\pm}])$. 
There is a decreasing filtration by the order of poles in $\LL$:
\begin{equation} \label{Polefilt} P^p \mathcal{M} = \{ f \in \mathcal{M}:    a^{(k)}_{m,n}(f) = 0  \quad  \hbox{ if } \quad  k<p\} 
\end{equation} 
which satisfies $P^a \mathcal{M} \times P^b \mathcal{M} \subset P^{a+b} \mathcal{M}$. 
In this paper we  shall work in  the subspace 
\begin{equation}\label{FirstQuadPoles}  \bigoplus_{r,s\geq 0} P^{-r-s} \mathcal{M}_{r,s}  \quad \subset \quad \mathcal{M} 
\end{equation}
with  non-negative modular weights  $(r,s)$ and poles in $\LL$  bounded by $w$.

\subsection{Differential  operators}  For all $r,s\in \Z, $ there exist operators 
$$
\partial_r : \mathcal{M}_{r,s}   \To    \mathcal{M}_{r+1, s-1} \qquad  , \qquad 
  \overline{\partial}_s : \mathcal{M}_{r,s}   \To   \mathcal{M}_{r-1, s+1}  \nonumber 
$$
which are a variant of the Maass raising and lowering operators, 
defined by 
\begin{equation} \partial_r = (z - \overline{z}) {\partial \over \partial z} + r  \qquad  , \qquad \overline{\partial}_s = (\overline{z} - z) {\partial \over \partial \overline{z}} + s \ .
\end{equation}
If we extend the action of $\partial_r, \overline{\partial}_s$ to all of $\mathcal{M}$ to be the zero map on all components $\mathcal{M}_{i,j}$ for $i\neq r$ or $j\neq s$ respectively, we obtain bigraded  differential operators
$$ \partial = \sum_r \partial_r : \mathcal{M} \To \mathcal{M} \qquad \hbox{ and } \qquad 
 \overline{\partial} = \sum_s \overline{\partial}_s : \mathcal{M} \To \mathcal{M}  $$
of modular weights $(1,-1)$ and $(-1,1)$ respectively.  They satisfy
$\partial(\LL) =\overline{\partial}(\LL) = 0$,
and generate a copy of $\ssl_2$ acting upon $\mathcal{M}$:
\begin{equation}
[ h ,\partial  ] =   2 \partial \qquad \ , \qquad  [ h , \overline{\partial}  ] = -   2 \overline{\partial} \qquad \ , \qquad [\partial, \overline{\partial}]=h\  
\end{equation} 
where $h : \mathcal{M} \rightarrow \mathcal{M}$
is multiplication by $r-s$ on the component $\mathcal{M}_{r-s}$ of $\mathcal{M}$.
In \cite{ZagFest} it is shown that   the kernels of the operators $\partial, \overline{\partial}$ are \label{propModularKernel} 
$$\mathrm{ker} (\partial_r : \mathcal{M}_{r,s} \rightarrow \mathcal{M}_{r+1,s-1} )  \quad  =   \quad \LL^{-r} \overline{M}_{s-r}$$
$$\mathrm{ker} (\overline{\partial}_s : \mathcal{M}_{r,s} \rightarrow \mathcal{M}_{r-1,s+1} )  \quad  =  \quad \LL^{-s} M_{r-s}$$
where $M_n$ and $\overline{M}_{n}$ denote the spaces of holomorphic and antiholomorphic modular forms, respectively.  In particular, 
$\ker \partial \cap \ker \overline{\partial} = \C[\LL^{\pm}]$.

\subsection{Laplace operator}
For any  $r,s \in \Z$, 
the Laplace operator 
$\Delta_{r,s} : \mathcal{M}_{r,s} \rightarrow \mathcal{M}_{r,s}$ is defined 
by the equivalent formulae  
\begin{equation}  \label{LaplaceDef}
 \Delta_{r,s}  \quad = \quad     - \overline{\partial}_{s-1} \partial_r   + r(s-1)   \quad = \quad    -  \partial_{r-1} \overline{\partial}_{s} +s(r-1)\ . 
  \end{equation} 
  The component   $\Delta_{0,0}$ is the Laplace-Beltrami operator.
 Let us denote by 
$\Delta : \mathcal{M} \rightarrow \mathcal{M}$
 the linear operator which acts by $\Delta_{r,s}$ on the component  $\mathcal{M}_{r,s}$. 
 
\subsection{Petersson inner product}
 For any integer $n$ the map
\begin{equation}  \label{IPpairing}
f \quad \times \quad  g  \qquad   \mapsto  \qquad \langle f, g\rangle :=  \int_{\mathcal{D}} f (z)   \overline{g(z)} \,  y^n  \, d\hbox{vol}  \nonumber 
\end{equation}
defines a pairing  $ \mathcal{M}_{r,s} \times \mathcal{S}_{n-s,n-r}   \rightarrow  \C $,
where
$$\mathcal{D} = \{ |z| >1 , |\mathrm{Re} (z) | <  \textstyle{1\over 2} \} \qquad \hbox{ and } \qquad d\hbox{vol} = \displaystyle{dx dy \over y^2}  $$
are the interior of the standard fundamental domain for the action of $\SL_2(\Z)$ on $\HH$, and  the $\SL_2(\Z)$-invariant volume form on $\HH$ in its standard normalisation. 

Restricting to holomorphic cusp forms $S_n \subset M_n$, we deduce pairings
$$ \langle f, g\rangle:    \mathcal{M}_{r,s} \times S_{r-s}  \To  \C \qquad \hbox{ and } \qquad  \langle f, 
\overline{g}\rangle:    \mathcal{M}_{r,s} \times \overline{S}_{s-r}  \To  \C \ .$$
 Since the Petersson inner product is non-degenerate on holomorphic cusp forms, these two pairings are equivalent to  a linear map
\begin{equation} \label{holprojection} p = (p^{h} , p^{a}) : \mathcal{M}_{r,s} \To  S_{r-s} \oplus \overline{S}_{s-r} \  , \end{equation}
 whose components are called the holomorphic and anti-holomorphic projections \cite{Sturm}.

\begin{thm} \label{thmpartialsorthogtocusp} Let $f \in \mathcal{M}_{r,s}$. If  $f = \partial F$ for some  $F\in  \mathcal{M}$ then 
\begin{equation} \langle f, g \rangle= 0  \qquad \hbox{ for all } \quad g \in S_{r-s} \hbox{ holomorphic} \  . \end{equation} 
In particular, $f$ is in the kernel of the holomorphic projection $(\ref{holprojection})$. 
\end{thm} 
This can be written  $p^h \partial  =0$. By taking the  complex conjugate, $p^a   \overline{\partial}   = 0 $.

\subsection{Example: real analytic Eisenstein series} \label{sectRAEisenstein}
Recall the functions  $\mathcal{E}_{r,s}  \in  P^{-w}\,  \mathcal{M}_{r,s}$  defined in \eqref{introEisdefn}. They satisfy the system of differential equations
 \begin{eqnarray} \label{realanalholequation} 
 \partial \, \mathcal{E}_{w,0}   &= & \LL \GE_{w+2}     \\ 
 \partial \,\mathcal{E}_{r,s} -(r+1) \mathcal{E}_{r+1, s-1}  &= &  0   \qquad \qquad  \hbox{ for all  }1\leq s\leq w \  \nonumber 
\end{eqnarray} 
and 
\begin{eqnarray} \label{realanalantiholequation} 
\overline{\partial} \, \mathcal{E}_{0,w}   &= & \LL  \overline{\GE}_{w+2}     \\ 
 \overline{\partial} \,\mathcal{E}_{r,s} - (s+1) \mathcal{E}_{r-1, s+1}   &= &   0  \qquad  \qquad \hbox{ for all  } 1\leq r \leq w \  .  \  \nonumber 
\end{eqnarray} 
They are eigenfunctions of the Laplacian: $\Delta \mathcal{E}_{r,s} =  -  w \, \mathcal{E}_{r,s}$, with constant part 
\begin{equation}  \label{constantpartofRealEis}
\mathcal{E}_{r,s}^0 =  {  - B_{w+2} \over  2(w+1)(w+2)} \LL      +      {(-1)^s \over 2} {w! \over 2^w}\binom{w}{r}     \zeta(w+1)  \LL^{-w}   
\end{equation}
This example is fairly atypical - in general the differential equations with respect to the holomorphic and anti-holomorphic differentials $\partial$ and $\overline{\partial}$ will not be symmetric. The theory is set up in  such a manner as to make the differential structure with respect to $\partial$ as simple as possible, at the cost of losing explicit control over the action of $\overline{\partial}$.   This is entirely analogous to the situation in genus $0$, since $\Lo$ satisfies a complicated differential equation with respect to $\frac{\partial}{\partial \overline{z}}$ whose coefficients involve multiple zeta values.

Let us write 
$$\mathcal{E}= \sum_{r+s=2w}  \mathcal{E}_{r,s} (X- \tau Y)^r (X-\overline{\tau }Y)^s \ . $$
It is equivariant: $\mathcal{E}(\gamma \tau) = \mathcal{E}|_{\gamma}$ for all  $\gamma \in \SL_2(\Z)$.  Consider   the equivariant 1-form 
\begin{equation} \label{Eunderlinefirstdef} \underline{E}_{2w+2}(\tau) = 2 \pi i \, \GE_{2w+2}(\tau) (X- \tau Y)^{2w} d\tau \ . 
\end{equation}
The normalisation of the power of $2 \pi i$ in this expression can be different in different contexts \cite{MMV}. 
 Equations $(\ref{realanalholequation})$ and $(\ref{realanalantiholequation})$ are equivalent to the differential equation:  
 $$d \mathcal{E}    =   { 1 \over 2}  \Big( \underline{E}_{2w+2}(\tau)   + \overline{\underline{E}_{2w+2}(\tau)}  \Big)  =   \mathrm{Re} \, \big(  \underline{E}_{2w+2}(\tau)  \big)  $$
The functions $\mathcal{E}_{r,s}$ are  thus obtained from the real part of an indefinite integral of  a holomorphic Eisenstein series. This is the prototype for the general theory. 

\section{Reminders on $\SL_2$-representations} \label{sectSL2}
Throughout this paper,  all tensors are over $\Q$ unless otherwise indicated.

\subsection{Definitions}
For all $n \geq 0$ define
$V_{2n} = \bigoplus_{r+s = 2n} X^r Y^s \Q$, 
equipped with the right action of $\SL_2$, and hence  $\SL_2(\Z)$, given by 
$$(X,Y)\big|_{\gamma} =   (aX+ bY, cX+ dY)$$
for $\gamma  $ of the form  $(\ref{gammaact})$. There is an isomorphism  of $\SL_2$-representations
$$V_{2m} \otimes V_{2n} \cong V_{2m+2n} \oplus V_{2m+2n-2} \oplus \ldots \oplus V_{2|m-n|}\ .$$
Consider the $\SL_2$-equivariant  projector 
\begin{equation}  \label{deltakdef}
\delta^k  : V_{2m} \otimes V_{2n} \To V_{2m+2n- 2k}
\end{equation} 
defined by 
$$\delta^k =   \mu   \circ \Big( {\partial \over \partial X} \otimes {\partial \over \partial Y} - {\partial \over \partial Y}\otimes {\partial \over \partial X}  \Big)^k  $$
where $\mu: \Q[X,Y] \otimes  \Q[X,Y]  \rightarrow  \Q[X,Y]$ is the multiplication map.

Recall the standard notation for generators of $\SL_2(\Z)$: 
\begin{equation} \label{SandTdef}
S= 
\left(
\begin{array}{cc}
  0   & -1  \\
   1  &   0 
\end{array}
\right)\quad, \quad  T= \left(
\begin{array}{cc}
  1   &  1  \\
   0  &   1 
\end{array}
\right)
\ .
\end{equation} 

\subsection{de Rham version} 
It is convenient to define a de Rham version $V_{2n}^{dR}$ of the (Betti) vector space $V_{2n}$ generated by elements denoted by $\Xv$ and $\Yv$. 

There is a comparison isomorphism
\begin{align}  \label{VBdRcomp}
V^{dR}_{2n} \otimes \C  \quad  &  \overset{\sim}{\To} \quad     V_{2n} \otimes \C   \\
(\Xv,\Yv)   \quad  &  \mapsto \quad      (X ,  (2 \pi i)^{-1} Y) \nonumber
\end{align}
The reason for this is that $X$ and $Y$  will span copies of $\Q(0)$ and $\Q(1)$ respectively: for example,  $X$ is a Betti basis for $\Q(0)$ and $\Xv$ a de Rham basis.

The vector space $V^{dR}_{2n}$ admits a (de Rham) action of $\SL_2(\Z)$ on the right via
$$(\Xv,\Yv)\big|_{\gamma} =   (a\Xv+ b\Yv, c\Xv+ d\Yv)$$
for $\gamma  $ of the form  $(\ref{gammaact})$. When the context is not clear, we shall denote this copy of $\SL_2$ by $\SL_2^{dR}$.  The comparison isomorphism  is an isomorphism of group schemes
$$
\comp_{B,dR}: \SL_2 \times \C  \overset{\sim}{\To}  \SL_2^{dR}\times \C $$
which on the level of points is given by conjugation by $ ( \begin{smallmatrix} 1&  0 \\ 0& (2  \pi i)^{-1} \end{smallmatrix})$:
\begin{equation}
\label{SL2dRversusB}
\comp_{B,dR} \begin{pmatrix} a& b \\ c& d \end{pmatrix}   \quad  =  \quad    \begin{pmatrix} a&   2 \pi  i b \\   (2 \pi i)^{-1}   c& d \end{pmatrix} 
\end{equation}
Consequently, the  Betti action of $\SL_2(\Z)$ on $V^{dR}_{2n}\otimes \C$ is twisted by powers of $2\pi i$. The images of  the Betti elements $S$ and $T$ under $\comp_{B, dR}$  are
\begin{equation} \label{SandTdeRhamdef}
S'= 
\left(
\begin{array}{cc}
  0   & - 2\pi i   \\
  (2\pi i)^{-1}  &   0 
\end{array}
\right)\quad, \quad   T'= \left(
\begin{array}{cc}
  1   &  2 \pi i   \\
   0  &   1 
\end{array}
\right)
\ .
\end{equation}  
If bar denotes complex conjugation, these satisfy $\overline{S'} =  - S'$  and $\overline{T'} = (T')^{-1}$.  

\begin{rem}  \label{remdeltadR}
There is a de Rham version of the projection  $\delta_{dR}: V^{dR}_{2m} \otimes V^{dR}_{2n}\rightarrow V_{2n+2m-k}$ defined in an identical manner
to (\ref{deltakdef}), except that we replace $X,Y$ with $\Xv, \Yv$.  Under the comparison isomorphism,  $\delta_{dR}^k$ corresponds to $(2\pi i)^k \delta^k$.
\end{rem}

\section{Iterated integrals of Eisenstein series}

\subsection{Preamble on filtrations} 
In order to keep the paper as accessible as possible, we work mostly with Lie algebras and iterated integrals and gloss over the geometric foundations.
From time to time, a paragraph marked with a star explains the geometric background with references for the interested reader.  Only in \S\ref{sect: ComplexConjugation}, which is considerably more technical, do we require any substantial  Hodge theory and Tannakian theory of fundamental groups and their completions.  As a result,  the 
objects described in this and later sections are equipped with a limiting mixed Hodge structure and  in particular,   possess three filtrations: $W, M$ and $F$.  Since all mixed Hodge structures considered here are of mixed Tate type, 
 the monodromy-weight filtration $M$ is split in the de Rham realisation by the Hodge filtration $F$, and is therefore associated to a grading  which we call the $M$-degree. It also determines $F$, which will not be discussed again. 
The geometric weight filtration $W$ plays a relatively minor  role in this paper. Indeed,  it is canonically split in the de Rham realisation, and the $W$-degrees can be deduced from the $M$-degrees and  $\ssl_2$-weights.  It is related to the depth filtration \cite{Sigma}, which is briefly discussed in \S\ref{sect: Relations}.  For this reason, we shall emphasize  the $M$-degree, and the interested reader can look up the $W$-filtrations in \cite{MEM, MMV}.

\subsection{A tensor algebra} 
 The generators $\Xv, \Yv$ of   $V^{dR}_{2n}$  satisfy
\begin{equation}\label{degsXY} \deg_M \Xv=0 \quad \ , \quad \deg_M \Yv=-1 \ . 
\end{equation}

Let $\e_{2n+2}$ be a symbol which corresponds to the Eisenstein series of weight $2n+2$. It is assigned  an $M$-degree of $-1$ (it spans a copy of $\Q(1)$).
Our convention is that typeface in sans serif font denote elements in a de Rham realisation.  

\begin{defn} Let $\uu^{dR}_E$ denote the free graded Lie algebra generated by elements
$$\e_{2n+2} \Xv^i \Yv^{2n-i}  \quad \hbox{ for } \quad  0 \leq i \leq 2n \ $$
and for all $n\geq 1$.  Equivalently, it is the free graded Lie algebra generated by $\e_{2n+2}V^{dR}_{2n}$, for every $n\geq 1$.
It is naturally equipped with a right action of $\SL^{dR}_2(\Z)$.
\end{defn} 
The completed Lie algebra $\widehat{\uu}^{dR}_E$ is the Lie algebra of a pro-unipotent affine group scheme $\U^{dR}_E$ over $\Q$.  We remark straight away that this 
object is not natural from a geometric perspective (it is not the de Rham realisation of a motive), but will be very convenient for computational   purposes. 
 Its affine ring is the tensor coalgebra
$$\Or(\U^{dR}_E) = T^c   \big(\bigoplus_{n\geq 1} \EE_{2n+2}\dV^{dR}_{2n} \big)$$
where $\EE_{2n+2}$ are symbols dual to $\e_{2n+2}$, and where for any graded vector space  $W$, the tensor coalgebra is defined to be the bigraded vector space
$$  T^c(W) = \bigoplus_{n\geq 0} W^{\otimes n} \ ,$$
equipped with the shuffle product $\sha$, and  the deconcatenation coproduct 
\begin{eqnarray} \Delta: T^c (W) & \To & T^c(W) \otimes T^c(W)    \nonumber \\  
\Delta (w_1\otimes \ldots \otimes w_n )  &=&  \sum_{0\leq i\leq n} (w_1 \otimes \ldots \otimes w_i) \otimes  (w_{i+1} \otimes \ldots \otimes w_n) \ . \nonumber 
\end{eqnarray} 
The ring $\Or(\U^{dR}_E)$ and  the group scheme $\U^{dR}_E$ inherit a left (respectively right) action of $\SL^{dR}_2$. 
Let us denote also by  $\widehat{\U}^{dR}_E$ the completed universal envelopping algebra of $\uu^{dR}_E$. It is the   ring of non-commutative
formal power series in $\e_{2n+2} \Xv^i \Yv^{2n-i}$, for $0\leq i\leq 2n$.   We can view the points of $\U^{dR}_E$ as the set of group like elements in $\widehat{\U}^{dR}_E$. 

\subsection{Power series connection} Recall (\ref{Eunderlinefirstdef}) that 
\begin{eqnarray}  \label{EunderlineasGE} \underline{E}_{2n+2}(\tau) & = &  2 \pi i \, \GE_{2n+2}(\tau) (X- \tau Y)^{2n} d\tau  \\
& \overset{(\ref{VBdRcomp})}{=} &  \GE_{2n+2}(q) (\Xv- \log (q)  \Yv)^{2n}  { dq \over q} \ .\nonumber 
\end{eqnarray} 
The second expression  has only rational coefficients in $q, \log(q)$.

Consider the following formal one-form taking values in $ \Omega^1(\HH)  \otimes \widehat{\U}^{dR}_E  $
\begin{equation} \label{OmegaEdef}
\Omega^{E}(\tau) = \sum_{n\geq 1 }   \underline{E}_{2n+2}(\tau) \e_{2n+2}
\end{equation} 
It is equivariant with respect to the (Betti) action of $\SL_2(\Z)$: 
\begin{equation} \nonumber
\Omega^E(\gamma  \tau ) \big|_{\gamma} = \Omega^E(\tau)  \quad \hbox{ for all } \quad  \gamma \in \SL_2(\Z) \ .
\end{equation} 
Here, and throughout this section,  the (Betti) action of $\SL_2(\Z)$ on the de Rham group $\U^{dR}_E(\C)$ is via the  comparison map (\ref{SL2dRversusB}).
Consider the trivial vector bundle $\widehat{\U}^E\otimes  \C$ over $\HH$, equipped with the connection 
$$\nabla  : \widehat{\U}^{dR}_E\otimes \C \To  \Omega^1(\HH)  \otimes \widehat{\U}^{dR}_E  $$
defined by $\nabla = d + \Omega^E$, 
where the elements $\e_{2n+2}$ act on $ \widehat{\U}^{dR}_E$ by concatenation on the left.  This connection is integrable since $d\,\Omega^E= \Omega^E \wedge \Omega^E= 0$.

\begin{prop} \label{propIE}  There exists a canonical horizontal section, 
$$I^E :  \HH \To  \widehat{\U}^{dR}_E(\C)$$
which is regularised at the unit tangent vector $\tone_{\infty}$ at the cusp.  It is given explicitly by the generating series of iterated integrals
\begin{eqnarray} I^E(\tau) &  = & 1 +  \int_{\tau}^{\tone_{\infty}} \Omega^E +    \int_{\tau}^{\tone_{\infty}} \Omega^E  \Omega^E + \ldots  \label{IEitint} \\ 
 &=& \sum_{r\geq0}  \sum_{n_1,\ldots, n_r \geq 1}  \e_{2n_1+2} \ldots \e_{2n_r+2}  \int_{\tau}^{\tone_{\infty}}  \underline{E}_{2n_1+2}(\tau) \ldots \underline{E}_{2n_r+2}(\tau)  \nonumber
 \end{eqnarray}
 where we integrate from the left.   It satisfies the following properties:
 \begin{eqnarray} 
(i) & & d I^E(\tau)  \   =  \ - \Omega^E(\tau) I^E(\tau)     \nonumber \\
(ii) & & I^E(\tone_{\infty}) \  =  \ 1  \nonumber \\
(iii) & & I^E(\tau)  \ \in \  \U^{dR}_E(\C) \qquad \hbox{ for all } \tau \in \HH  \ .\nonumber 
\end{eqnarray} 
Furthermore, the coefficient  of $\e_{2n_1+2} \ldots \e_{2n_r+2} \Xv^{i_1} \Yv^{2n_1-i_1} \otimes \ldots \otimes \Xv^{i_r} \Yv^{2n_r-i_r}  $  in $I^E(\tau)$ is an element of the ring
$\C [[q]][\log (q) ] $
of degree at most $2(n_1+\ldots +n_r) +r$  in $\log(q)$. 

\end{prop} 
\begin{proof} The definition of the tangential base point $\tone_{\infty}$ is given in \cite{MMV} \S2, as well as a concrete explanation of the regularised iterated integrals.  Parts $(i)$, $(ii)$ and $(iii)$ are proved in \cite{MMV}, propositions 1  and 2. Only the last part remains to be proven.  For this, observe that the
ring $\C[[q]][\log (q) ] $ is closed under primitives with respect to $q^{-1} dq $, i.e., if 
$f\in \C[[q]][\log (q) ] $, then there exists $F\in \C[[q]][\log (q) ]$ such that $dF = q^{-1}f dq$.
 Furthermore, 
  a  primitive of $ q^i \log^j (q) \, dq$ is again a polynomial in $q$ and $\log q$ of degree at most $j+1$ in $\log(q)$.  Since the coefficients of $\underline{E}_{2n+2}(\tau)$ lie in $\C[[q]][\log(q)]dq$ and have degree at most $2n$ in $\log q$, the statement follows.
\end{proof} 
Properties $(i)$ and $(ii)$ in the previous proposition determine $I^E(\tau)$ uniquely. Property $(iii)$ is equivalent to a shuffle product relation between iterated integrals of Eisenstein series, which is spelled out  in \cite{MMV} (3.8). 

\begin{cor}
For every $\gamma \in \SL_2(\Z)$, there is a unique  $\CC^E_{\gamma} \in \U^{dR}_E(\C)$ such that
\begin{equation} 
I^E(\gamma \tau ) \big|_{\gamma} \CC^E_{\gamma} = I^E(\tau)
\end{equation}
for all $\tau \in \HH$. It satisfies the cocycle equation 
\begin{equation}
\CC^E_{ g h } = \CC^E_g \big|_{h} \CC^E_h \qquad \hbox{ for all } g , h \in \SL_2(\Z) \ .
\end{equation} 
\end{cor}
\begin{proof}  Identical to \cite{MMV}  lemma 5.1, bearing in mind that the right action of $\SL_2(\Z)$ on $\U^{dR}_E(\C)$ is given by (\ref{SL2dRversusB}).
\end{proof} 

The map  $\gamma \mapsto \CC^E_{\gamma}$ defines a non-abelian cocycle 
$$\CC^E  \quad \in \quad Z^1 (\SL_2(\Z), \U^{dR}_E(\C)) \ .$$
 Since $\SL_2(\Z)$ is finitely presented, and  generated by the two elements $T, S$, any cocycle $\CC$
is uniquely determined (\cite{MMV} lemma 5.5) by its values $\CC_S$ and $\CC_T$, which satisfy
\begin{eqnarray}
\CC_S\big|_S \, \CC_S  &= &1  \nonumber \\
\CC_U\big|_{U^2}  \, \CC_U\big|_U \, \CC_U & = & 1 \ , \nonumber 
\end{eqnarray} 
where $U= TS$ and $\CC_U = \CC_T\big|_S \CC_S$. 
 The  series $\CC^E_T$ was computed explicitly in \cite{MMV} \S6, and its coefficients  involve only  powers of $2 \pi i$. The coefficients of $\CC^E_S$  are  numbers which we called multiple modular values in \cite{MMV}.

\subsubsection{Geometric background*}
The affine group scheme $\U^{dR}_E$ is not a good object in the sense that it does not admit a natural mixed Hodge structure.
A closely related object  is  the de Rham relative completion  of the fundamental group of $\mathfrak{M}_{1,1}$ with respect to the unit tangent vector at the cusp \cite{HaGPS, MMV}.  The group scheme $\U^{dR}_E$ is a  quotient  of it (but not in the category of mixed Hodge structures).  The reason for working with  $\U^{dR}_E$ is for simplicity of exposition, since under the monodromy homomorphism,  discussed below, all generators of  relative completion  which correspond to cusp forms act trivially, leaving only the Eisenstein series. These  are artificially captured by $\U^{dR}_E$.

\subsection{Totally holomorphic iterated integrals} In this section we work over the ring of complex numbers, and all tensors are over $\C$. 
Consider the tensor coalgebra $$T^c \big(\bigoplus_{n \geq 1} M_{2n+2}(\C)\otimes V^{\vee}_{2n} \big)$$ where $M_{2n+2}(\C)$ denotes the complex vector space of holomorphic modular forms of weight $2n+2$. Elements in this coalgebra define  `totally holomorphic' iterated integrals on $\HH$. They   are linear combinations of the following iterated integrals  
\begin{equation} \label{itintgeneral} \int_{\tau}^{\tau_0}  \omega_1 \ldots \omega_r 
\end{equation}
where $\omega_k $ are of the form  $(2\pi i) \tau^{m_k} f_k(\tau)d\tau $ with $f_k\in M_{2n_k+2}(\C)$ and $0\leq m_k\leq 2n_k$. 
For $\tau_0$ a finite base point these were considered by Manin  Manin \cite{Ma1, Ma2}.
 The previous expression is the  coefficient of $\pm X^{2n_1-m_1}Y^{m_1}\otimes \ldots \otimes  X^{2n_r- m_r}Y^{m_r}$ in 
$$\int_{\tau}^{\tau_0} \underline{f_1}(\tau) \ldots \underline{f_r}(\tau)\ ,$$
where for any modular form $f \in M_{2n+2}(\C)$ of weight $2n+2$, we write
$$\underline{f}(\tau) =  2 \pi i f(\tau) (X - \tau Y)^{2n} d \tau$$
A different normalisation of the power of $2\pi i$ was used in the first chapter of  \cite{MMV}. 
From the definition of iterated integrals one has:
$${ d \over d \tau} \int_{\tau}^{\tau_0}  \omega_1 \ldots \omega_r 
 = - \omega_1    \int_{\tau}^{\tau_0}  \omega_2 \ldots \omega_r \ . $$
The corresponding integrals   regularised with respect to a tangent vector at the cusp were studied in \cite{MMV} \S4.

\begin{prop} Let $\tau_0$ be any (possibly tangential) basepoint  on $\HH$.  The  iterated integration  homomorphism (with respect to the shuffle product on the tensor coalgebra): 
$$ \int_{\tau}^{\tau_0} : T^c(\bigoplus_{n \geq 1} M_{2n+2}(\C)\otimes V^{\vee}_{2n}) \To \hbox{ holomorphic functions   of } \tau \in  \HH$$
is injective. 
It follows that iterated integrals are linearly independent over the ring  $M[\tau]$ generated by  holomorphic modular forms and the function $\tau$.
\end{prop} 
We give two different proofs of linear independence. The first is conceptual and uses properties of relative completion. It is a generalisation to vector-valued iterated integrals of a well-known theorem on  the linear independence of iterated integrals due to Chen \cite{Ch}.  The second  proof is elementary but more computational.\footnote{Since writing this paper, a third approach   in the case of Eisenstein series appeared  in \cite{Matthes}.}

We first deduce the second part of the proposition from the linear independence. 

\begin{proof} Since holomorphic modular forms are an algebra, we can multiply any linear relation amongst iterated integrals with coefficients in $M[\tau]$  by a non-zero modular form of sufficiently large weight to obtain a relation 
 of the form 
$$\sum_{I=( i_1,\ldots, i_n)}   \lambda_{I} \omega_{i_1} \int^{\tau_0}_{\tau} \omega_{i_2} \ldots \omega_{i_n} =0 $$
where $\lambda_I \in \C$ and every $\omega_i $ is of the form $\tau^k f(\tau)$ where $k+2$ is bounded above by the weight of $f$.
This in turn implies, by integrating, a linear relation of the form 
$$\sum_{I=( i_1,\ldots, i_n)}   \lambda_{I} \int^{\tau_0}_{\tau}  \omega_{i_1} \omega_{i_2} \ldots \omega_{i_n}  = \lambda $$
where $\lambda \in \C$. Since the right-hand side is a multiple of the empty iterated integral $1$, this is relation between iterated integrals. Using the fact that they are linearly independent, we deduce that  all coefficients $\lambda$ vanish. 
\end{proof} 

The following corollary follows immediately from the proposition by applying it term by term in $\overline{q}, \overline{\tau}$. 
\begin{cor} \label{corLIoverModforms}  Let  $\mathcal{I} \subset \C[[q]][\tau]$ denote the $\C$-vector space generated by the iterated integrals
$(\ref{itintgeneral})$. Then the following  natural map  is an injection:
$$ \mathcal{I} \otimes M[\tau]  \otimes  \overline{\mathcal{I}} \otimes  \overline{M}[ \overline{\tau}] \To \C [[q,\overline{q}]][\tau, \overline{\tau}]\ .$$ 
\end{cor}

\subsubsection{General proof of linear independence using relative completion} 

 Suppose that  $\omega \in  T^c(\bigoplus_{n \geq 1} M_{2n+2}(\C)\otimes V^{\vee}_{2n})$ such that 
$$\int_{\tau}^{\tau_0} \omega = 0$$
for all $\tau \in \HH$. In particular,  for any $\gamma \in \SL_2(\Z)$, the integral from $\tau=\gamma^{-1} \tau_0$ to $\tau_0$ also vanishes.
It follows that $\omega$ vanishes on $\pi_1(\mathfrak{M}_{1,1}(\C), x_0) \cong \SL_2(\Z)$, where $\mathfrak{M}_{1,1}(\C)$ is the orbifold quotient of $\HH$ by the action of $\SL_2(\Z)$, and $x_0$ is  the image of $\tau_0$.  But 
$$T^c \big(\bigoplus_{n \geq 1} M_{2n+2}\otimes V^{\vee}_{2n}\big) \quad \subset \quad \Or( \pi_1^{\rel, dR}(\mathfrak{M}_{1,1}(\C), x_0))$$
where $\pi_1^{\rel}$ denotes the de Rham version of   relative completion \cite{HaGPS, HaMHS, MMV}.  The topological fundamental group $\pi^{\mathrm{top}}_1(\mathfrak{M}_{1,1}(\C), x_0)$ is Zariski-dense in  its group-theoretic relative completion which is the Betti realisation of  the de Rham relative completion, and isomorphic to it over $\C$ via the comparison isomorphism. Thus $\SL_2(\Z)$ is Zariski-dense in the latter, and we have $\omega =0$, i.e, 
 iterated integrals are linearly independent.  Note that the  proof works more generally for the whole of the affine ring of de Rham relative completion, and not just for totally holomorphic iterated integrals.

\subsubsection{An elementary proof of linear independence}
For the benefit of the reader not familiar with relative completion, we spell out the above proof of 
 linear independence of iterated integrals in elementary terms. It only 
uses the Eichler-Shimura theorem.

Fix $\tau_0 \in \HH$. 
Given a modular form $f$ of weight $w$ and $0\leq i \leq w-2$, the differential one form $\tau^i f(\tau)d\tau$  on $\HH$ defines 
a function 
$$
 \gamma \mapsto  \int_{\gamma} \tau^i f(\tau) d\tau   \quad  : \quad  \SL_2(\Z)  \To    \C  
$$
where $\gamma$ denotes the geodesic path from $ \gamma^{-1} \tau_0$ to $ \tau_0$.

\begin{lem} Let $f_1,\ldots, f_n$ be linearly independent  holomorphic modular forms of weights $w_1,\ldots, w_n$. Then the functions
$\tau^i f_j(\tau) d\tau: \SL_2(\Z) \rightarrow \C$, where $1\leq j \leq n$ and $0\leq i \leq w_j-2$, are linearly  independent over $\C$.
\end{lem}
\begin{proof}
An element $g\in \SL_2(\Z)$  defines an automorphism of $\HH$. By functoriality of integration with respect to automorphisms (change of variables formula),  
a  linear combination of one-forms of the above type $\omega$ vanishes on all $\gamma \in \SL_2(\Z)$ if and only if the same is true of $g^* \omega$.
Now observe  by modularity of $f$ that 
\begin{eqnarray}  S^*  (f(\tau) d\tau)  &=  & \tau^{w-2} f(\tau) d\tau  \nonumber \\
T^*   (\tau^i f(\tau)  d\tau) & = &  (\tau+1)^i f(\tau) d\tau \ ,  \nonumber
\end{eqnarray} 
where $w$ is  the weight of $f$, and $T \tau = \tau+1$, $S \tau  = -\tau^{-1}$.  It follows that 
$$(T^* - \id)   (\tau^i f(\tau)  d\tau)  = i \tau^{i-1}  f(\tau) d\tau  \ +  \ \hbox{  terms of lower order in  } \tau^j f(\tau)$$
Equivalently, the complex vector space spanned by the  $\tau^i f(\tau) d\tau$ for $ 0\leq i \leq w-2$  is an  irreducible $\SL_2(\Z)$-representation, 
and  any non-zero vector in it generates the entire space under the action of  $S$ and $T$.   The following arguments are easily translated into representation-theoretic language, but we give a long-winded account for explicitness. 
Suppose that there exists an   $\omega=\sum \lambda_{ij} \tau^i f_{j}(\tau) d\tau $ which vanishes on $\SL_2(\Z)$ for some $\lambda_{ij} \in \C$ not all zero.   By applying $(T^*-\id)^N$ for a sufficiently large $N$, we can  assume that all powers of $\tau$ are zero, i.e., $\lambda_{i,j}=0$ for $i>0$.  By applying $S^*$, and once again  $(T^*-\id)^{N'}$ for some $N'$,   we can assume in addition  that all $f_j$ have the same weight $w$. Thus we can write $\omega=  \sum \lambda_j f_j(\tau) d\tau $ where $\lambda_j \in \C$ are not all zero.   Via the action of $g^*$ for $g \in \SL_2(\Z)$, we deduce that  $\tau^i \omega$ vanishes as a function on $\SL_2(\Z)$  for all $0\leq i \leq w-2$. The same is therefore true of the  formal linear combination
$$ \sum_j \lambda_j f_j(\tau) (X- \tau Y)^{w-2} d\tau \ $$
 which has coefficients in $V_{2n}$.
 But this is equivalent to the statement 
 $$ \sum_j \lambda_j C_{f_j}=0$$
 where for $f \in M_{2n+2}(\C)$ holomorphic, $C_{f} \in Z^1(\SL_2(\Z), V_{2n}\otimes \C)$ is the cocycle 
 $$ C_f =  \gamma \mapsto  \int^{\tau_0}_{\gamma^{-1}\tau_0}  f(\tau) (X- \tau Y)^{w-2}$$
   associated to $f$. But the Eichler-Shimura theorem  \cite{Eichler, Shimura} implies that the map 
   $$ f\mapsto [C_f] \ : \ M_{2n+2}(\C) \To H^1(\SL_2(\Z), V_{2n}\otimes \C)$$
   is injective, so  \emph{a fortiori} the cocycles $C_{f_j}$  are independent.  Hence $\lambda_j=0$ for all $j$, a contradiction.
\end{proof}

The following lemma is an easy exercise, but included for completeness. 
\begin{lem} Given any linearly independent functions $\phi_1,\ldots, \phi_m: S \rightarrow \C$, where $S$ is a set, there exist $x_1,\ldots, x_m \in  \bigoplus_{s\in S} \C s
$ such that $\phi_i(x_j) = \delta_{ij}$ for all $1\leq i,j \leq m$, where the  functions $\phi_i$ are extended  linearly to linear combinations in $S$.
\end{lem}
\begin{proof} 
Consider the matrix-valued function  $P= (\phi_j(s_i))_{i,j}$  on $S^{m}$.  Its determinant is not identically zero  on $S^{m}$.
For otherwise, a row expansion would yield a relation $\sum_{i=1}^{m} \lambda_i \phi_i(s_m) =0$, for all $s_m$, where the 
$\lambda_i$ are minors of $P$. By linear independence,  the $\lambda_i$ vanish. The same argument can then be applied to each minor of $P$, and we eventually deduce by induction on the size of $P$ that every entry $\phi_j(s_i)$ vanishes for all $s_i$, a  contradiction.    Let $s_1,\ldots, s_m \in S$ such that $\det (P)\neq 0$. Let $e_k$ be the column vector with a $1$ in the kth row and zeros elsewhere. Let $\lambda_k$, for $1\leq k \leq m$, denote the vector in $\C^m$ such that $\lambda_k^T   P  = e_k$.  Set  $x_k = (s_1, \ldots ,s_m). \lambda_k  \in  \bigoplus_{s\in S} \C s$. By construction it satisfies $\phi_i(x_j)= \delta_{ij}$. 
\end{proof}

Suppose that we have a non-trivial relation  of minimal length $n$ for all $\tau \in \HH$  $$\sum_{I} \lambda_I \int_{\tau}^{\tau_0} \omega_{i_1}\ldots \omega_{i_n} + \big(\hbox{iterated interals of  length} \leq n-1 \big)=  0 \ ,$$
where the $\omega_{i_j} \in  \{\omega_1,\ldots, \omega_N\}$, a set  of  linearly independent 1-forms  of the kind $\tau^i f(\tau) d\tau$ where $0 \leq i \leq w-2$, where $f \in M_{w+2}(\C)$.   Composition of paths  for iterated integrals \cite{Ch} states that 
$$\int_{\alpha \beta} \omega_{i_1} \ldots \omega_{i_n} = \sum_{k=0}^n \int_{\alpha} \omega_{i_1} \ldots \omega_{i_k}  \int_{\beta} \omega_{i_{k+1}} \ldots \omega_{i_n} \ $$
where $\alpha$, $\beta$ are two composable paths, and $\alpha \beta$ denotes the path $\alpha$ followed by $\beta$.  By linearity, it holds more generally for $\alpha$ any linear combination of paths which can be composed with $\beta$. 
By the previous two lemmas, there exist   elements $x_i \in \C[\SL_2(\Z)]$ for $1\leq i \leq N$, where $\SL_2(\Z)$ are classes of paths based at $\tau$,  such that 
$$\int_{x_j} \omega_i = \delta_{ij}\ .$$
 Apply the composition of paths formula with $\beta$ a path from $\tau$ to $\tau_0$, and $\alpha=   x_j$.
Since $\alpha\beta$ is  a linear combination of paths from points in the $\SL_2(\Z)$-orbit of $\tau$ to $\tau_0$ and since  the above linear combination of iterated integrals vanishes  along each such path,  we 
deduce $N$ relations, for each $1 \leq i \leq N$,  of the form
$$\sum_I  \lambda_I  \delta_{ i_1 j}  \int_{\tau}^{\tau_0} \omega_{i_2} \ldots \omega_{i_n}   + \big(\hbox{iterated interals of  length} \leq n-2 \big)= 0   \   $$
which have length $\leq n-1$.  By minimality, the coefficients $\lambda_I$ such that $i_1=j$ vanish. Since this holds for all 
$j=1,\ldots, N$, we conclude that the $\lambda_I$ all  vanish, a contradiction.

\section{A Lie algebra of geometric derivations}

 \subsection{Geometric context*} The de Rham fundamental group
 $  \pi^{dR}_1(\mathcal{E}^{\times}_{\partial/\partial q}, \tone_{0})$ of the punctured first order smoothing of the Tate curve \cite{MEM, HaGPS},  with  the basepoint given by (a choice of)  unit tangent vector  at the origin, is the de Rham realisation of a mixed Tate motive over $\Z$  \cite{HaMTZ}. This last fact is not strictly required for this paper and could be circumvented at the cost of complicating some arguments.  By standard arguments, the de Rham realisation is canonically graded and admits  an isomorphism 
 \begin{equation} \label{Pidefn}  \Pi  \cong   
  \pi^{dR}_1(\mathcal{E}^{\times}_{\partial/\partial q}, \tone_{1})\end{equation}
  where $\Pi$ is the pro-unipotent affine group scheme whose  Lie algebra is the completion of the bigraded Lie algebra $\Lie(\av, \bv)$ discussed below.

\subsection{Derivations}

Consider the    free bigraded Lie algebra $\Lie(\av,\bv)$
  on two generators $\av, \bv$.
We shall only be concerned with its $M$-grading, which  satisfies
$$\deg_M (\av) = -1 \qquad  , \qquad \deg_M(\bv) = 0 $$
Recall that this is one half of the $M$-grading in \cite{MEM, MMV}. 

 Let $\mathrm{Der}^{\Theta} \Lie(\av,\bv)$ denote the Lie algebra of derivations 
$$\delta : \Lie(\av, \bv) \To \Lie(\av, \bv)$$
which annihilate  the element $\Theta=[\av, \bv] \in \Lie(\av,\bv)$, i.e., 
$$[\delta(\av), \bv] + [\av, \delta(\bv)]=0\ .$$

\begin{propdefn} There exists a distinguished family of derivations \cite{Tsunogai, Nakamura, Pollack, MEM} for every $n\geq 0$:
\begin{equation} \label{varepsilondefn} \varepsilon_{2n+2}^{\vee}  \quad  \in \quad   \mathrm{Der}^{\Theta} \Lie(\av,\bv)\ ,
\end{equation} 
which are uniquely determined by the property $\varepsilon_{2n+2}^{\vee} \Theta=0$,    the formula 
$$\varepsilon_{2n+2}^{\vee} (\av)  \quad  =  \quad  \mathrm{ad} (\av)^{2n+2} (\bv) \ ,$$
 and the fact that $\varepsilon^{\vee}_{2n+2}(\bv)$ is of degree $\geq 1$ in $\bv$.  Their action on $\bv$ is given  by 
$$\varepsilon_{2n+2}^{\vee}  (\bv) \quad = \quad  {1 \over 2} \sum_{i+j= 2n+1} (-1)^i [ \mathrm{ad}(\av)^i \bv, \mathrm{ad}(\av)^j \bv]\ , $$
where $i,j\geq 0$. 
These derivations were first written down by Tsunogai \cite{Tsunogai}.
\end{propdefn}

The Lie algebra $\Lie(\av, \bv)$ admits a right action by $\SL^{dR}_2(\Z)$
\begin{eqnarray} \Lie(\av, \bv) \times \SL^{dR}_2(\Z) & \To &  \Lie(\av, \bv) \label{SL2actsonLab}  \\ 
(\av, \bv) \big|_{\gamma}  & =  &    (d\, \av  +  c\, \bv , b\,  \av + a\, \bv ) \nonumber 
 \end{eqnarray} 
 where $\gamma$ is given by $(\ref{gammaact})$. The infinitesimal version of this action gives rise to an action of the Lie algebra $\ssl_2$ via the  following derivations:
$$\varepsilon_0  = - \av  { \partial \over \partial \bv}    \qquad , \qquad   \varepsilon_0^{\vee} = \bv {\partial \over \partial \av}   \ . $$
The notation is consistent with $(\ref{varepsilondefn})$ since $\varepsilon^{\vee}_0$ is indeed the case $n=0$ of (\ref{varepsilondefn}).
Since they  annihilate $\Theta$, they generate a copy of $\ssl_2 $ inside $ \mathrm{Der}^{\Theta} \Lie(\av,\bv)$, via
$$[\varepsilon_0, \varepsilon_0^{\vee} ]= \hh \quad , \quad  [\hh, \varepsilon_0]= -2 \varepsilon_0 \quad , \quad  [\hh, \varepsilon^{\vee}_0]=   2\varepsilon_0^{\vee}     $$
where  $\hh \in \mathrm{Der}^{\Theta} \Lie(\av,\bv)$ is multiplication by $\deg_{\bv} - \deg_{\av}$. 
The derivation algebra $ \mathrm{Der}^{\Theta} \Lie(\av,\bv)$ therefore  admits an inner action of $\ssl_2$. 
With our conventions,  a derivation  $x \in  \mathrm{Der}^{\Theta} \Lie(\av,\bv)$ is a lowest weight vector if and only if it satisfies
$[\varepsilon_0, x]=0$. It is a highest weight vector if and only if $[\varepsilon_0^{\vee}, x]=0$.

\begin{lem} \label{lemsl2onepsilons}  The elements $\varepsilon_{2n+2}^{\vee}$ are highest-weight vectors and generate an irreducible representation of $\ssl_2$ of dimension $2n+1$, with the basis  $\mathrm{ad}^{i} (\varepsilon_0^{\vee}) \varepsilon^{\vee}_{2n+2}$ for $0\leq i \leq 2n$.
   In particular, $\mathrm{ad}^{2n+1}(\varepsilon_0^{\vee})(\varepsilon^{\vee}_{2n+2})=0$.
\end{lem}
\begin{proof} These facts are proved in \cite{MEM}. 
\end{proof} 

Let us denote by $\varepsilon_{2n+2}$ the operators obtained from $\varepsilon^{\vee}_{2n+2}$ by conjugating by the element $S$, where $(\av,\bv)\big|_S  =(-\bv, \av)$. These are lowest weight vectors and satisfy 
$$\varepsilon_{2n+2} (\bv) \quad = \quad  - \mathrm{ad} (\bv)^{2n+2}(\av)\ . $$

\subsection{(de Rham) Lie algebra of geometric derivations}

\begin{defn} Define a Lie subalgebra 
$$\ue  \leq  \mathrm{Der}^{\Theta} \Lie(\av,\bv)$$ to be the bigraded Lie subalgebra generated by the $\varepsilon_{2n+2}^{\vee}$ for $n \geq 1$, together with  the action of $\ssl_2$.  
By the previous lemma, it is the bigraded Lie subalgebra  generated by the derivations $\mathrm{ad}(\varepsilon_0^{\vee})^i \varepsilon_{2n+2}^{\vee}$ for all $n\geq 1$. 
\end{defn}

With this definition, which differs slightly from the version given in \cite{Sigma}, $\ue$ does not contain the element  $\varepsilon^{\vee}_2=\varepsilon_2$, 
which plays a limited role, and   commutes (as one may verify) with all elements of $\ue$. 
Therefore    the Lie algebra generated by all
derivations $\varepsilon_{2n+2}^{\vee}$ for $n\geq -1$  is    isomorphic to 
$$\varepsilon_2\Q \oplus ( \ssl_2 \ltimes  \ue) \  .$$

\begin{defn} Let $\Ue$ denote the (de Rham) pro-unipotent affine group scheme whose Lie algebra is the completion of  $\ue$.  
It admits a right action of $\SL^{dR}_2(\Z)$. 
\end{defn}

\subsection{Relations}  \label{sect: Relations} The elements $\varepsilon_{2n+2}$ satisfy many non-trivial relations, which were first studied by Pollack  \cite{Pollack} following a suggestion of Hain (in Pollack's notation, the notations  $\varepsilon$ and $\varepsilon^{\vee}$ are reversed). 
One shows that the elements 
\begin{equation}  \label{indepelements} [ \varepsilon_0^i \varepsilon_{2a+2} , \varepsilon_0^j \varepsilon_{2b+2}]  \qquad \hbox{ where } \quad  0\leq i \leq 2a \  , \  0\leq j \leq 2b
\end{equation} 
are linearly independent when $a,b \geq 1$ and   $a+b \leq 4$, but starting from  $M$-degree $\leq -12$,  Pollack showed that there are relations
\begin{eqnarray} \label{epsilonrel1}  [\varepsilon^{\vee}_{10}, \varepsilon^{\vee}_4 ] -  3 [\varepsilon^{\vee}_8, \varepsilon^{\vee}_6]  &= & 0  \\ 
2[ \varepsilon^{\vee}_{14}, \varepsilon^{\vee}_{4}] - 7[\varepsilon^{\vee}_{12}, \varepsilon^{\vee}_{6}] + 11[\varepsilon^{\vee}_{10}, \varepsilon^{\vee}_{8}]  &= &  0 \ , \nonumber
\end{eqnarray} 
which are the first two  in an infinite sequence of quadratic and higher order relations. 
There are several different proofs of these relations in the literature. We shall derive a new  and elementary  interpretation of these relations 
via the  orthogonality of equivariant iterated integrals of Eisenstein series to cusp forms  (theorem \ref{thmpartialsorthogtocusp}).

\subsection{Embeddings} One way to think of  the Lie algebra $\ue$ is as a quotient of a certain free Lie algebra generated by Eisenstein symbols (see the next section). Another is to embed it inside the free Lie algebra $\Lie \, (\av, \bv)$.

Let us write $H= \Q \av \oplus \Q \bv$. The natural map 
$$\mathrm{Der} \, \Lie(\av,\bv) \To   \mathrm{Hom}( H,  \Lie(\av, \bv) )$$
is injective, since a derivation $\delta$ is uniquely determined by $\delta(\av), \delta(\bv)$. Hence
$$\ue \To H^{\vee} \otimes_{\Q} \Lie(\av, \bv)$$
is also injective. In fact, more is true. 
It is straightforward to show that, as a consequence of the relation $\delta(\Theta)=0$,  the action of a derivation on either $\av$ or $\bv$ defines a pair of injective linear maps
\begin{equation} \label{evx}
{ev}_x: \delta \mapsto \delta (x):  \ue \To \Lie (\av, \bv)
\end{equation} 
where $x =\av$ or $\bv$. 
These maps do not respect the Lie algebra structure but  can be used to write down elements in  $\Or(\Ue)$. We shall only use them in \S\ref{sect: PLS}.

\section{Monodromy homomorphism}

\subsection{Geometric preamble*}  Let $\mathcal{G}^{dR}_{1,\veco}$ denote the de Rham   relative  completion \cite{HaMHS, MMV} of the fundamental group of the moduli scheme  of elliptic curves equipped with a non-zero abelian differential, with base point the unit tangent vector   at the cusp.  The monodromy action defines a canonical  homorphism 
$$\mathcal{G}^{dR}_{1, \veco} \To \mathrm{Aut} ( \pi_1^{dR}( \mathcal{E}^{\times}_{\partial/\partial q}, \tone_0))\ .$$
After choosing suitable splittings of the weight filtrations, it gives rise to a morphism from the graded Lie algebra of 
$\mathcal{G}^{dR}_{1, \veco}$  to $\varepsilon_2 \Q \oplus \ue$ ($\varepsilon_2$ is central). One  deduces the existence of  a homomorphism
$$\mathcal{G}^{dR}_{1,1} \To \Ue$$
where $\mathcal{G}^{dR}_{1,1}$ is the de Rham relative completion of the fundamental group of $\mathfrak{M}_{1,1}$. One knows \cite{MMV, MEM} that  the previous map factors through its quotient  $\U^{dR}_E$, i.e., all generators corresponding to cusp forms map to zero.

\subsection{Description of the map $\mu$}

The isomorphism
$$(\Xv,\Yv) \mapsto (\bv,\av) :  \Q \Xv\oplus \Q \Yv \overset{\sim}{\To}  \Q \bv \oplus \Q \av$$
  respects the action of $\SL^{dR}_2(\Z)$ on both sides. This induces an isomorphism of  the de Rham Lie algebras $\ssl_2$ which act on $\uu^{dR}_E$ and $\ue$ respectively:
\begin{equation} \label{XYderivs}
\Xv {\partial  \over  \partial \Yv } \mapsto  \varepsilon_0^{\vee} \qquad \hbox{ and } \qquad -  \Yv {\partial  \over  \partial \Xv } \mapsto  \varepsilon_0 \ .
\end{equation} 

\begin{defn}  \label{defnuEtoue} There is a unique  morphism of Lie algebras satisfying 
\begin{eqnarray} \label{mudefn}
\mu:   \uu^{dR}_E & \To & \ue  \\
\e_{2n+2} \Xv^{2n}  &  \mapsto & {2 \over (2n)!} \,  \varepsilon_{2n+2} \nonumber
\end{eqnarray} 
for all $n\geq 1$, which is equivariant for the respective actions of $\ssl_2$ on both sides.   It respects the $M$ grading. 
\end{defn} 
  The map $\mu$ is surjective.
 The map $\mu$ induces a surjective  (faithfully flat) homomorphism of affine group schemes 
$$\mu : \U^{dR}_E \To\!\!\!\!\!\!\!\!\!\To \Ue\ .$$

\subsection{The generating series $J$}
\begin{defn}
Let  $J$ denote the composition  $\mu  I^E$.  It defines an analytic  function 
$$ J : \HH \To \Ue(\C) $$
whose coefficients are certain linear combinations of iterated integrals of $\GE_{2n}$, $n\geq 2$.
\end{defn}
The following proposition  is perhaps known to one or two  experts, but  is not stated explicitly in the literature to our knowledge. 

\begin{prop}
The function $J$ is the unique solution to  the differential equation 
\begin{equation} \label{dJ}  d J  =  - \omega J 
\end{equation}
taking the value $1$ at the tangent vector $\partial/\partial q$ at  $q=0$,  where  $\omega$ is the formal one-form
\begin{equation} \label{omegadef} \omega =  -\mathrm{ad}(\varepsilon_0)  {dq \over q}   + \sum_{n \geq 1}  {2 \over (2n)!}  \varepsilon_{2n+2} \GE_{2n+2}(q) {dq \over q}\ .
\end{equation} 
\end{prop}

\begin{proof}
It follows from the identity 
$$(\Xv- \log(q) \Yv)^n = \exp \Big(-\log(q)  \Yv {\partial \over \partial \Xv} \Big) \Xv^n$$
and the definition   \eqref{mudefn} of $\mu$ applied to  (\ref{EunderlineasGE}) that 
\begin{eqnarray} \mu (\e_{2n+2} \underline{E}_{2n+2}(\tau))  &= &  {2 \over (2n)!}  \exp \big(  \log(q) \mathrm{ad}(\varepsilon_0) \big) \varepsilon_{2n+2} \GE_{2n+2}(q) {dq \over q}  \nonumber \\
& = & {2 \over (2n)!}\sum_{m\geq 0 }  \frac{1}{m!}  (\log q)^m   \mathrm{ad}(\varepsilon_0)^m \varepsilon_{2n+2}  \GE_{2n+2}(q) {dq \over q}  \nonumber \\
&= & {2 \over (2n)!} \sum_{m\geq 0}\varepsilon_{2n+2}^{(m)}    \Big( \int^{\tone_\infty}_q  \underbrace{{dq \over q} \ldots {dq \over q}}_m \Big) \GE_{2n+2}(q) {dq \over q} \ , \nonumber
\end{eqnarray} 
where for simplicity we write $\varepsilon_{2n+2}^{(m)}$ for $( \mathrm{ad}( -\varepsilon_0))^m  \varepsilon_{2n+2}$. These elements  vanish whenever $m> 2n$  as a consequence of  lemma \ref{lemsl2onepsilons}.  

The tangential basepoint $\tone_{\infty}$ corresponds, in the $q$-disc, to $ \frac{\partial}{\partial q}$,  the unit tangent vector at the origin. 
From the formula \eqref{IEitint} for $I^E$ as an iterated integral transposed to the $q$-disk, we see that its image $\mu I^E$ has the form
$$ J = \mu I^E = \sum_{r\geq 0 } \sum_{n_{1}, \ldots, n_r\geq 0 } \sum_{m_1,\ldots, m_r\geq 0 }    J^{m_1,\ldots, m_r}_{n_1,\ldots, n_r} (q)$$
where 
\begin{multline} J^{m_1,\ldots, m_r}_{n_1,\ldots, n_r}(q)  = {2 \over (2n_1)! } \ldots {2 \over (2n_r)!} \varepsilon_{2n_1+2}^{(m_1)}  \ldots   \varepsilon_{2n_r+2}^{(m_r)}  
\\
\times 
\int_q^{\frac{\partial}{\partial q}}   \Big( \underbrace{{dq \over q} \ldots {dq \over q}}_{m_1}  \GE_{2n_1+2}(q) {dq \over q} \Big)  \ldots \Big( \underbrace{{dq \over q} \ldots {dq \over q}}_{m_r} \GE_{2n_r+2}(q) {dq \over q}\Big)
\end{multline} 
 It follows from the formula for the derivative of an  iterated integral that
$$ d J^{m_1,\ldots, m_r}_{n_1,\ldots, n_r}(q)  \quad   = \quad     - \mathrm{ad}( - \varepsilon_0)  {dq \over q} \, J^{m_1-1, m_2, \ldots, m_r}_{n_1,n_2, \ldots, n_r}(q)  $$
if $m_1 \geq 1$, and 
$$ d J^{0, m_2,\ldots, m_r}_{n_1,\ldots, n_r}(q)  \quad   = \quad   -   {2 \over (2n_1)!}  \varepsilon_{2n_1+2} \GE_{2n_1+2}(q)   {d q\over q}    \, J^{m_2,\ldots, m_r}_{n_2,\ldots, n_r}(q)  $$
otherwise. Therefore $J$ is 
the unique solution to the differential equation $(\ref{dJ})$ which takes the value $1$ at the unit tangent vector at the origin. 
\end{proof} 
The image of  $\CC^E$ under the homomorphism $\mu$ is a non-abelian cocycle
$$ \CCu := \mu \,\CC^E  \quad \in \quad Z^1 (\SL_2(\Z) ,   \Ue(\C) )\ ,$$
where $\SL_2(\Z)$ acts via  the Betti action    (\ref{SL2dRversusB}). 
 The function $J$ satisfies  \begin{equation} \label{Jmonodromy}
  J(\gamma \tau ) \big|_{\gamma}  \,  \CCu_{\gamma}  = J(\tau) 
 \end{equation} 
 for all $\tau \in \HH$, and $\CCu_{\gamma}$ satisfies the cocycle equation 
 \begin{equation} \label{CCucocycle}
\CCu_{gh} = \CCu_g\big|_{h}  \CCu_h 
 \end{equation} 
  for all $g, h \in \SL_2(\Z)$. 

\section{Action of complex conjugation} \label{sect: ComplexConjugation}
Complex conjugation acts via $q\mapsto \overline{q}$  on the unit disc, and  corresponds on the upper half plane to the involution $\tau\mapsto -\overline{\tau}$. 

 It acts on $\Ue(\C)$ by complex conjugation on the coefficients. 
It is important to note that it does not respect the cocycle relation (\ref{CCucocycle}), since
$$\overline{\CCu}_{gh} = \overline{\CCu}_g\big|_{\overline{h}} \, \overline{\CCu}_h $$
and $\overline{h}$ is not in general equal to $h$ since 
 the (Betti) action of $\SL_2(\Z)$ on $\Ue(\C)$ is via  (\ref{SL2dRversusB}), which involves powers of $2\pi i$ (see (\ref{SandTdeRhamdef})). To remedy this, we compose  the action of complex conjugation with the element  $-1 \in \G_m(\Q)$, where $\G_m$ is the multiplicative group corresponding to the $M$-grading.  
 
 More precisely, consider the involution 
 $$ (\av, \bv) \mapsto (-\av, \bv) : \LL(\av, \bv) \To \LL(\av, \bv)$$
 It has the effect of multiplying terms of $M$-degree $m$ by $(-1)^m$.  
This induces an involution $(-1): \ue \rightarrow \ue$ on the Lie subalgebra $\ue$, and hence an involution
$(-1) : \Ue \rightarrow \Ue$ on affine group schemes. 

\begin{defn}
Consider  the homomorphism  of groups 
\begin{equation} \label{sdef} \sv: \Ue(\C) \To \Ue(\C)
\end{equation}
which is obtained by composing $(-1) : \Ue \rightarrow \Ue$ with the action of complex conjugation on coefficients, in either order (they commute). 
\end{defn} 

It follows that the involution $\sv$ now preserves the cocycle relation:
 $$
\sv\, \CCu_{gh} = \sv \! \left(\CCu_g\big|_{h} \right) \sv \left( \CCu_h  \right)
 $$
  for all $ g, h \in \SL_2(\Z)$, 
and hence $\sv \,\CCu \in Z^1(\SL_2(\Z), \Ue(\C))$. It was shown in \cite{MMV} that the space of cocycles is a torsor over a certain automorphism group, and the technical heart of this paper is to relate $\CCu$ and $\sv\,  \CCu$ via this automorphism group.
To state this, 
recall that 
$\ZZ^{\sv} \subset \R$ 
denotes the ring of   single-valued multiple zeta values.
The following theorem involves a considerable amount  of the machinery constructed in \cite{MMV}, and may be taken as a black box.

\begin{thm} \label{thmExistbphi} There exist elements 
$$b^{\sv}  \in \Ue(\ZZ^{\sv}) \qquad \hbox{ and } \qquad \phi^{\sv} \in \mathrm{Aut}(\Ue)^{\SL_2}(\ZZ^{\sv})$$
such that, for all $\gamma \in \SL_2(\Z)$, we have 
\begin{equation} \label{bphionCgamma} (b^{\sv})^{-1}\big|_{\gamma} \phi^{\sv}( \sv \, \CCu_{\gamma} )  b^{\sv} = \CCu_{\gamma} \ .
\end{equation}
The  pair of elements $(b^{\sv}, \phi^{\sv})$ is well-defined  up to replacing it with $( ab^{\sv} , a \phi^{\sv} a^{-1})$, where $a\in \Ue(\ZZ^{\sv})^{\SL_2}$.
Furthermore, the automorphism  $$ (b^{\sv})^{-1} \phi^{\sv} b^{\sv}  \quad \in \quad  \mathrm{Aut}(\Ue)(\ZZ^{\sv})$$ is induced by an automorphism   $\psi^{\sv} \in \mathrm{Aut}(\Pi)(\ZZ^{\sv})$ of $\Pi$ (defined in (\ref{Pidefn})).
\end{thm}
The group $ \mathrm{Aut}(\Ue)^{\SL_2}$ denotes the $\SL_2$-invariant automorphisms of $\Ue$, i.e., 
automorphisms $\phi: \Ue \overset{\sim}{\rightarrow} \Ue$ satisfying $\phi ( x g) = \phi(x) g$ for all $ g\in \SL^{dR}_2$, with  the de Rham action  on $\Ue$.  Therefore, an element $\phi \in  \mathrm{Aut}(\Ue)^{\SL_2}$ commutes with the (Betti)  image  of $\SL_2(\Z)$ in $
\SL_2^{dR}(\C)$ under the comparison map (\ref{SL2dRversusB}).

\subsection{Further properties} \label{SectFurther} Before turning to the proof, we state  a number of further  properties satisfied by the elements $b^{\sv}, \phi^{\sv}$ of the theorem.

(i).  The elements $b^{\sv}, \phi^{\sv}$ are respectively exponentials of elements
$$\beta^{\sv} \in \ue (\ZZ^{\sv}) \quad \hbox{ and } \quad \delta^{\sv} \in \mathrm{Der} ( \ue)^{\ssl_2}(\ZZ^{\sv}) $$
where the derivation $\delta^{\sv}$ has  the property that 
$$[\delta^{\sv}, \varepsilon_0] =0 \qquad \hbox{ and } \qquad [\delta^{\sv}, \varepsilon^{\vee}_0] =0  $$
which  is equivalent to saying that  $\delta^{\sv}$ is $\ssl_2$-invariant. 
\\

(ii). There exists a  derivation 
$$\pp^{\sv}=\Lie \,\psi^{\sv}    \quad \in \quad    \mathrm{Der}^{\Theta} \Lie(\av,\bv)(\ZZ^{\sv}) \ $$
which preserves $\ue(\ZZ^{\sv}) $ and whose restriction to   $\ue(\ZZ^{\sv}) $  satisfies
$$\pp^{\sv} \big|_{\ue(\ZZ^{\sv}) } =  \mathrm{ad}(\beta^{\sv}) + \delta^{\sv} \ .$$
This is a highly non-trivial constraint.  For example, it expresses the non-obvious fact that the derivation on the right hand side 
is uniquely determined by its action on a single element $\av$ (or $\bv$). Conversely, the fact that $\pp^{\sv}$ in turn normalises the image of all geometric derivations $\ue$ is equivalent to an infinite sequence of  combinatorial constraints of the form
$$ [  \pp^{\sv} , \varepsilon_{2n+2} ] \in \ue (\ZZ^{\sv})  $$
for every $n\geq 1$.  It is far from obvious that there  exist any solutions to these equations but follows from 
the Tannakian  theory implicit in the proof of theorem \ref{thmExistbphi} (we know that `motivic' derivations must satisfy a similar property). 
\\

(iii). The equation $(\ref{bphionCgamma})$, in the case $\gamma=T$, is equivalent to the  `inertial relation'
$$[\beta^{\sv}, \varepsilon_0] + [\beta^{\sv}, N_+]+ \delta^{\sv}(N_+) \quad = \quad 0 $$
where the element $N_+ \in \ue$ is the element of $M$-degree $-1$ given by 
$$N_+\quad =  \quad \sum_{n\geq 1}  {B_{2n+2}\over 4n+4} {2 \over (2n)!}  \varepsilon_{2n+2} $$
and is the unipotent part of the logarithm of $(T, \CCu_T) \in \SL_2\ltimes \Ue(\C)$. 
The inertial relation ties together $\beta^{\sv}$ and $\delta^{\sv}$:  information about $\beta^{\sv}$ can be deduced
from information about $\delta^{\sv}$ and vice-versa. 
The equation  $(\ref{bphionCgamma})$ is uniquely determined by its two instances $\gamma =T$ and $\gamma =S$. In other words, $(b^{\sv},\phi^{\sv})$ are uniquely determined (up to twisting by $a\in (\Ue)^{\SL_2}$ via $(b,\phi) \mapsto (ab^{\sv}, a \phi^{\sv} a^{-1})$), by the inertial relation and 
$$ ( b^{\sv})^{-1}\big|_{S} \phi^{\sv}( \sv \,  \CCu_{S} )  b^{\sv} = \CCu_{S} \ .$$

(iv). It follows as a consequence of \S18 of \cite{MMV} that to lowest order
$$b \equiv 1 + \sum_{n\geq 1}  \zeta_{\sv}(2n+1) \, \varepsilon^{\vee}_{2n+2}  \qquad \pmod{[\ue(\ZZ^{\sv}), \ue(\ZZ^{\sv})]} \ ,$$
is canonical, where $\zeta_{\sv}(2n+1) = 2 \zeta(2n+1)$.  The element $\phi^{\sv}$ is also known to lowest order by  \cite{MMV}, theorem 16.9, via the inertial relation. In particular, $$[\delta^{\sv}, \ue(\ZZ^{\sv})] \quad \subset \quad  [\ue(\ZZ^{\sv}), \ue(\ZZ^{\sv})]\ .$$ 
The element $\pp^{\sv}$   satisfies 
$\pp^{\sv} (\av) \equiv \av$ and $\pp^{\sv} (\bv) \equiv \bv$ modulo terms of degree $\geq 3 $ in $\bv$. 
\\

(v).   The weights of  the coefficients of $a$, $b^{\sv}$, $\phi^{\sv}$ are determined by the $M$-filtration. In fact, much more is true: replacing $b^{\sv}, \phi^{\sv}$ with their `motivic' versions as in \cite{MMV}, \S18.3, the action of the de Rham motivic Galois group of $\MT(\Z)$ on their coefficients is determined by its action on $(b^{\sv},\phi^{\sv})$. Restricting to the subgroup $\G_m$  implies that the weights of the motivic multiple zeta values are induced by the $M$-grading on $\Or(\Ue)$. 
\\

(vi). Since complex conjugation is an involution, the elements $\sv$ and hence $(b^{\sv}, \phi^{\sv})$ satisfy an involution equation which we will not write down.

\begin{rem} \label{remHainMorph} The element  $\psi^{\sv} \in \mathrm{Aut} \, \Pi(\C)$ is image of the single-valued element  $\sv$ defined below, and could  be computed independently  via the periods of $\Pi$. 
In particular, it is compatible  \cite{MMV} with the Hain morphism  $\Phi$  from the motivic fundamental group of the projective line minus three points $\Pro^1\backslash \{0,1,\infty\}$. The image of $\sv$ in its group of automorphisms was computed in \cite{SVMP} and involves single-valued multiple zeta values.  This provides yet another constraint, and hence a method to obtain information about $b^{\sv}$ and $\phi^{\sv}$, which we will not exploit here. 
\end{rem}

 In conclusion,
the elements $b^{\sv}, \phi^{\sv}$  are  very heavily constrained and it should in principle be possible to compute them explicitly to higher orders.

\subsection{Proof of theorem \ref{thmExistbphi}} 
The proof follows closely the argument given in \cite{MMV}, \S19 with some minor differences.   We summarize the main ingredients, and refer to loc. cit. for further details. 

\begin{enumerate} 

\item The objects $\GG^{dR}_{1,1}, \pi_1^{dR}(\Eq,\tone_0)$ are affine group schemes over $\Q$ whose affine rings are the de Rham components of Ind-objects in a category $\mathcal{H}$ of realisations. The objects  in $\mathcal{H}$ consist of triples
$(M_{B}, M_{dR}, c)$
where $M_B, M_{dR}$ are finite-dimensional $\Q$-vector spaces, and $c$ is an isomorphism $M_{dR} \otimes_{\Q} \C \overset{\sim}{\rightarrow} M_B \otimes_{\Q} \C$. In addition, $M_B, M_{dR}$ are equipped with an increasing weight filtration (in our situation, the $M$-filtration) compatible with $c$, and $M_{dR}$ has a decreasing filtration $F$ such that $M_B$,  equipped with the weight filtration and the filtration $c(F)\otimes \C$  on $M_B\otimes_{\Q} \C$ is a  graded polarisable $\Q$-mixed Hodge structure. The final piece of data is  a real Frobenius involution $F_{\infty}: M_B \overset{\sim}{\rightarrow}M_B$ compatible with $c$
and with complex conjugation on both $M_B \otimes \C$ and $M_{dR} \otimes \C$. 

\vspace{0.05in}
\item By a Tannakian argument \cite{MEM} Appendix B, one can choose splittings of the weight and Hodge filtrations $M$ and $F$ in the de Rham realisation of $\mathcal{H}$. They  are compatible with the monodromy homomorphism $\mu$.  Therefore we can identify
$\pi_1^{dR}(\Eq,\tone_0)$ with $\Pi$. The image of $\U_{1,1}^{dR}$, the unipotent radical of $\GG^{dR}_{1,1}$,  under  $\mu$ is by definition $\Ue$: 
$$\mu : \U_{1,1}^{dR} \To  \Ue \leq  \mathrm{Aut}(\Pi)\ .$$
\item Let $G^{dR}_{\mathcal{H}}$ denote the  group of tensor automorphisms of the   fiber functor on the category $\mathcal{H}$  which sends $(M_B, M_{dR}, c)$ to $M_{dR}$. It is an affine group scheme over $\Q$. 
Consider the composition of isomorphisms
$$M_{dR} \otimes     \C \overset{c}{\To} M_B \otimes    \C \overset{F_{\infty}}{\To} M_B \otimes  \C \overset{c^{-1}}{\To} M_{dR} \otimes \C\ .$$
Since it is functorial in $M$, it defines a canonical element 
$$\mathbf{s} \in G^{dR}_{\mathcal{H}}(\C)\ . $$
 Thus for every object $X$ in $\mathcal{H}$, we deduce the existence of $\mathbf{s} \in \mathrm{Aut}(X^{dR})(\C)$ which is compatible with all morphisms in $\mathcal{H}$, and computes the action of the real Frobenius $F_{\infty}$ in the de Rham realisation.

\vspace{0.05in}
\item The element $\mathbf{s}$ is canonical, but it is convenient to modify it as follows. Here our presentation differs slightly from that of \cite{MMV}, \S19.  The action of $G^{dR}_{\mathcal{H}}$  on the de Rham component  of the Lefschetz  object $\Q(-1) = (\Q, \Q, 1 \mapsto 2 \pi i )$ in $\mathcal{H}$
defines a morphism $\pi: G^{dR}_{\mathcal{H}} \rightarrow \mathrm{Aut}(\Q)=  \G_m$. The image of $\mathbf{s}$ under $\pi$ is $-1 \in \C^{\times} =  \G_m(\C)$. Now, the
 choice of splitting of the weight filtration $M$ is equivalent to an action of the  multiplicative group $\G_m$ on the de Rham component of objects of $\mathcal{H}$, i.e., a splitting of the  homomorphism $\pi: G^{dR}_{\mathcal{H}} \rightarrow \G_m $. We now  multiply  $\mathbf{s}$ by the image of  $-1 \in \G_m(\Q)$ under this splitting  to obtain a  modified element  $\sv = (- 1) \mathbf{s} \in G^{dR}_{\mathcal{H}}(\C)$,  which now acts by the identity on $\Q(-1)$. The element $\sv$ depends on the choice of splitting. 
 
 Note that  since $\pi_1(\Eq
 \tone_0)$ has  a mixed Tate Hodge structure, its $M$-filtration is canonically split in the de Rham realisation by $F$, and the action of $\sv$ upon its complex points is  canonical.   The same applies for $\Ue$, and so the statement of the theorem depends in no way on the choices of splittings.

\vspace{0.05in}
\item We therefore deduce the existence of an element 
$$\sv \in \mathrm{Aut}(\GG^{dR}_{1,1})(\C) \times \mathrm{Aut}(\Pi)(\C)$$
which is compatible with $\mu$.  It is the image of the element $\sv \in G^{dR}_{\mathcal{H}}$, which acts compatibly on 
both $\GG^{dR}_{1,1}(\C)$ and $\Pi(\C)$. 
On the other hand, in \cite{MMV} \S10, we gave an explicit description of the group of automorphisms of $\GG^{dR}_{1,1}$.
We showed that,  for any choice of splitting $\GG^{dR}_{1,1} \cong \SL_2 \ltimes \U^{dR}_{1,1}$, any  automorphism  of  $\GG^{dR}_{1,1}$ which acts trivially on $\SL_2$ defines a pair 
$$b \quad \in \quad \U^{dR}_{1,1}(\C) \qquad \hbox{ and } \qquad  \phi \quad \in \quad \mathrm{Aut}(\U^{dR}_{1,1})^{\SL_2}(\C)\ .$$
They are well-defined up to twisting by an element $a\in (\U^{dR}_{1,1})^{\SL_2}(\C)$. 
The (right) action of $(b, \phi)$ on $(g, u) \in (\SL_2 \ltimes \U^{dR}_{1,1})(\C)$ is  $(g, b^{-1}\big|_{g} \phi(u) b)$. 
 There exists, by the argument of  \cite{MEM} Appendix B,  a  splitting 
 $\GG^{dR}_{1,1} \cong \SL_2 \ltimes \U^{dR}_{1,1}$ compatible with the choice of $M$-splittings.   We deduce that the image of $\sv$ under the monodromy homomorphism $\mu$ is represented by a pair 
$$ b^{\sv} \quad \in \quad \Ue(\C) \qquad \hbox{ and } \qquad \phi^{\sv} \quad \in \quad \mathrm{Aut}(\Ue)^{\SL_2}(\C)\ ,$$
which are well-defined up to twisting  by $a \in (\Ue)^{\SL_2}(\C)$. 
It follows from the compatibility of $\sv$ with $\mu$ that the image of $\sv \in \mathrm{Aut}(\Ue)(\C)$ is given by the  automorphism $(b^{\sv})^{-1} \phi^{\sv} b^{\sv} $. It is induced by the automorphism $\sv \in \mathrm{Aut}(\Pi)(\C)$, which we call $\psi^{\sv}$ in the statement of the theorem. 
 
 \vspace{0.05in}
\item Now let us apply the element $\sv$ to the `canonical cocycle'. The Betti component $\GG^B_{1,1}$ of the relative completion $\GG_{1,1}$ admits a natural map 
$$\pi_1^{\mathrm{top}} (\mathcal{M}_{1,1}(\C), \tone_{\infty}) = \SL_2(\Z) \To \GG^B_{1,1}(\Q)\ . $$
This is one of the defining properties of relative completion. We deduce a map 
$$\SL_2(\Z) \To \GG^B_{1,1}(\Q) \overset{c}{\To} \GG^{dR}_{1,1}(\C) \cong \SL^{dR}_2(\C) \ltimes \U^{dR}_{1,1}(\C)$$
where $c$ denotes  the comparison isomorphism $\comp_{B,dR}$ for short. The image of $\gamma \in \SL_2(\Z)$ is 
$(c\gamma,  \mathcal{C}_{\gamma})$,  
where $c\gamma$ is the image of $\gamma$ under (\ref{SL2dRversusB}) and $\mathcal{C}_{\gamma} \in   \U^{dR}_{1,1}(\C)$ is called the canonical cocycle. Its image under $\mu$ is precisely $\CCu_{\gamma}\in \Ue(\C)$.  
Since $\gamma$ is Betti-rational, the action of $F_{\infty}$, corresponds, via the comparison isomorphism, to complex conjugation on coefficients of $c(\gamma)$. This action is computed by the element $\mathbf{s}$. The element $\sv$ computes complex conjugation composed with  the map $-1\in \G_m$ (\ref{sdef}).   Since the affine ring of $\SL_2$ is pure Tate, the latter action is trivial on $\SL_2$ and hence $\sv$ acts trivially on $\gamma$ (this is the reason for preferring $\sv$ over $\mathbf{s}$, which does not).  It follows that   the (right) action of $\sv$ satisfies the equation 
$$ (\gamma,\sv\,   \mathcal{\CCu}_{\gamma})  \circ \sv=   (\gamma,  \mathcal{\CCu}_{\gamma})\ .$$
Writing $\sv$ in terms of $(b^{\sv}, \phi^{\sv})$ above, we deduce that 
$$(b^{\sv})^{-1}\big|_{\gamma}  \phi^{\sv}( \sv \,\CCu_{\gamma})  b^{\sv} =  \CCu_{\gamma}\ .$$

\vspace{0.05in}
\item It remains to compute the coefficients of $\sv$. For this we need the fact that the motivic fundamental group of the punctured Tate elliptic curve (or rather, the image of $\GG_{1,1}$ under the monodromy homomorphism $\mu$) is a mixed Tate motive over $\Z$ \cite{HaMTZ}. It then follows from the results of \cite{SVMP} that, for any mixed Tate motive $M$ over $\Z$, the coefficients of $\sv$  in $\mathrm{Aut}(M_{dR})(\R)$ are single-valued multiple zeta values. 
\end{enumerate}

The  idea of this proof   applies in a  much more general setting.   It may be possible to circumvent the final step (7), which appeals to deep results about the category of mixed Tate motives over $\Z$ by a direct argument following the procedure outlined in remark \ref{remHainMorph}.  However, it is likely that this would forfeit some of the constraints described in \S\ref{SectFurther}, which would become conjectures given the current state of knowledge.

\section{Equivariant iterated Eisenstein integrals} \label{sect: EquivEisen}
We can now define modular-equivariant  iterated integrals of Eisenstein series.

\begin{defn} \label{defnJeqv} Let $(b^{\sv},\phi^{\sv})$ be as in   theorem \ref{thmExistbphi}.  Then define
$$J^{\eqv}(\tau)  =   J (\tau)   (b^{\sv})^{-1}  \phi^{\sv} (\sv J(\tau)^{-1}) \ .$$
It is well-defined up to right multiplication by an element $a\in (\Ue)^{\SL_2}(\ZZ^{\sv})$. 
\end{defn}

\begin{thm}  \label{thmJeqv} The series $J^{\eqv}$ defines a real analytic function 
$$J^{\eqv}: \HH \To \Ue(\C)$$
which is equivariant for the action of $\SL_2(\Z)$:
$$ J^{\eqv}(\gamma\tau )\big|_{\gamma} = J^{\eqv}(\tau) \ .$$
\end{thm}
\begin{proof} Using the monodromy equation $(\ref{Jmonodromy})$, we compute 
\begin{eqnarray} J^{\eqv}( \gamma\tau) \big|_{\gamma}  &  = &     J (\gamma\tau) \big|_{\gamma} \,  (b^{\sv})^{-1}|_{\gamma}\,  \phi^{\sv} (\sv J (\gamma\tau)^{-1}\big|_{\gamma} ) \nonumber \\
& = & J(\tau)  \CCu_{\gamma}^{-1}  \,  (b^{\sv})^{-1} |_{\gamma}\, \phi^{\sv}( \sv \, \CCu_{\gamma} \,  \sv \, J(\tau)^{-1})  \nonumber \\
& = &  J(\tau)  \CCu_{\gamma}^{-1}  \,  \Big( (b^{\sv})^{-1}|_{\gamma}\,   \phi^{\sv}( \sv \, \CCu_{\gamma})  b^{\sv} \Big) (b^{\sv})^{-1}  \phi^{\sv}( \sv \, J(\tau)^{-1}) \nonumber \\
& = & J(\tau) (b^{\sv})^{-1}  \phi^{\sv}( \sv \, J(\tau)^{-1}) \nonumber \\
& = & J^{\eqv}(\tau) \ . \nonumber 
\end{eqnarray}
The cancellation of terms going from  the third  to the fourth equation follows from the defining property $(\ref{bphionCgamma})$.
\end{proof}
The series $J^{\eqv}$ is of the form 
$$J^{\eqv}  =  J(\tau)  \left( \sv  J(\tau)\right)^{-1}  +  \hbox{ correction terms} \ ,$$
since $b^{\sv}$ and $\phi^{\sv}$ are to lowest order equal to $1$ and the identity respectively. 
Therefore, to leading orders, the coefficients of $J^{\eqv}$ are the real or imaginary parts  of iterated integrals of Eisenstein series of the appropriate length. 
The correction terms involve  linear combinations of products of real and imaginary parts of iterated integrals of Eisenstein series of lower lengths with    single-valued multiple zeta value coefficients. 

\subsection{A digression: `single-valued' versus `equivariant'}
The series $J^{\eqv}(\tau)$ does not quite correspond to the `single-valued version' of $J(\tau)$, which is 
$$J^{\sv}(\tau) := J (\tau)   (b^{\sv})^{-1}  \phi^{\sv} (\sv\, J(\tau)^{-1}) b^{\sv} =  J^{\eqv}(\tau) b^{\sv} \ .$$
Unlike $J^{\eqv}$, it is canonically defined, i.e., does not depend on the choice of representative $(b^{\sv}, \phi^{\sv})$. Its value at the unit tangent vector at the cusp is one:
$$J^{\sv}(\tone_{\infty})=1\ .$$
The generating function $J^{\sv}(\tau)$ is not modular equivariant. In fact, by the previous theorem it satisfies
\begin{eqnarray} 
J^{\sv}(\gamma \tau)\big|_{\gamma}  &= & J^{\eqv}(\gamma \tau)\big|_{\gamma}  b^{\sv}\big|_{\gamma}  \nonumber \\
 & =  & J^{\eqv}(\tau) b^{\sv}\big|_{\gamma} \nonumber  \\
 & = & J^{\sv}(\tau) (b^{\sv})^{-1} b^{\sv}\big|_{\gamma} \nonumber 
\end{eqnarray} 
Thus if we define 
$$\CCu_{\gamma}^{\sv} =  (b^{\sv})^{-1} b^{\sv}\big|_{\gamma}    \quad \in \quad Z^1(\SL_2(\Z) , \Ue(\C))$$
to be the single-valued cocycle, the single-valued generating series $J^{\sv}$ satisfies
$$J^{\sv}(\gamma \tau) \big|_{\gamma} = J^{\sv}(\tau) \CCu_{\gamma}^{\sv}\ .$$
The coefficients of $\CCu_{\gamma}^{\sv}$ are single-valued multiple zeta values since this is true for  $b^{\sv}$. A key point is that the cocycle $\CCu^{\sv}$ is a coboundary. It is for this reason that it can be trivialised (non-canonically) to 
produce a modular equivariant function $J^{\eqv}(\tau)$.

\subsection{Properties of $J^{\eqv}$} Let $(b^{\sv},\phi^{\sv})$ be as in theorem \ref{thmExistbphi}.

\begin{cor}  \label{corPropertiesofJeq} The generating function $J^{\eqv}$ satisfies the differential equation 
\begin{equation}  \label{JeqvODE} d J^{\eqv} =  - \omega J^{\eqv}   + J^{\eqv}  \phi^{\sv}( \sv \,\omega)\ .
\end{equation} 
It is the unique solution whose value at the tangent vector of length one at the cusp is 
$$J^{\eqv}( \tone_{\infty} )    =  (b^{\sv})^{-1} \ .  $$
\end{cor}
\begin{proof} The differential equation follows from the definition \ref{defnJeqv} together with the observation that 
$J^{-1}$ satisfies the equation $d J^{-1} = J^{-1} \omega$. This follows from differentiating the equation $ J^{-1} J=1$, which implies that 
$(dJ^{-1} ) J - J^{-1} \omega J =0$, and by multiplying on the right by $J^{-1}$.  The formula for the value of $J^{\eqv}$ at $\tone_{\infty}$ is a consequence of the definition, the fact that $J(\tone_{\infty})=1$, and the fact that $\phi^{\sv}, \sv$ are group  homomorphisms, and therefore preserve the identity in $\Ue(\ZZ^{\sv})$. 
\end{proof}

In particular, the holomorphic component of the differential  equation 
$$ \Big({\partial   \over \partial \tau} J^{\eqv}\Big) d\tau = -\omega  J^{\eqv} $$ is canonical, but the anti-holomorphic part
$$\Big({\partial \over   \partial \overline{\tau}}  J^{\eqv}  \Big) d\overline{\tau} =    J^{\eqv}  \phi^{\sv}( \sv\, \omega)$$ depends on the choice of $\phi^{\sv}$. 
The right-hand side is $J^{\eqv} \sv\,  \omega$ to leading order, and so
$  d J^{\eqv} \equiv -\omega J^{\eqv} + J^{\eqv} \sv\,  \omega$ modulo lower order terms.

\begin{lem}  \label{lemComplexConj} Complex conjugation acts on $J^{\eqv}$ via the formula:
$$\sv ( b^{\sv} J^{\eqv} ) =       \widetilde{\phi}^{\sv}  (J^{\eqv} b^{\sv})^{-1}   \ ,$$
where $\widetilde{\phi}^{\sv} =  (-1) \phi^{\sv} (-1)$. 
\end{lem} 
\begin{proof}  The element $K= \sv\, J^{\eqv}$ satisfies the equation 
$$ dK = - \sv\, \omega K + K \sv\, \phi^{\sv} \sv(\omega)$$
and hence 
$$ dK^{-1} = - \sv \,\phi^{\sv}\, \sv(\omega) K^{-1} + K^{-1} \sv( \omega) \ .$$
The element $F= \sv\, \phi^{\sv} \,\sv(J^{\eqv})$ satisfies the equation 
$$ dF =  - \sv \phi^{\sv} \sv (\omega) F + F \sv \phi^{\sv} \sv \phi^{\sv} \sv (\omega)\ .$$
Since the holomorphic components of these two differential equations agree, 
it follows that $F = K^{-1} A$, for some $A: \HH \rightarrow \Ue(\C)$ which  satisfies 
$$A(\gamma \tau )\big|_{\gamma}= A(\tau) \qquad \hbox{ and } \qquad  
{\partial A \over \partial \tau} =0\ .$$
Its modular equivariance follows from the equivariance of $\phi^{\sv}$ and the fact that $\sv$ preserves $\SL_2(\Z)$. 
Consider any coefficient $a: \HH \rightarrow V_{2n}\otimes \C$ of $A$.
It  defines a holomorphic section 
$\overline{a}  \in  \Gamma(  \SL_2(\Z) \backslash \! \!\backslash  \HH ;    \mathbb{V}_{2n})$, where $\mathbb{V}_{2n}$ denotes the vector bundle associated to $V_{2n}$, which is analytic   at the cusp.  Such sections correspond to  modular forms of weight $2n$. 
On the other hand, from $A= KF$ we see that the coefficients of $A$ are linear combinations of iterated integrals of modular forms and their complex conjugates. By  corollary \ref{corLIoverModforms},  $A$ is constant.
We deduce that 
$$\sv (J^{\eqv}) = A\,  \sv \, \phi^{\sv} \,  \sv (J^{\eqv})^{-1}$$
Multiplying by $\sv(b^{\sv})$ on the left and  changing the constant $A$, this  is equivalent to 
$$ \sv  (b^{\sv} J^{\sv} (b^{\sv})^{-1}) = A' \sv\,\phi^{\sv}\, \sv (J^{\sv})^{-1} = A' \widetilde{\phi}^{\sv}  (J^{\sv})^{-1}$$
where $J^{\sv} = J^{\eqv} b^{\sv}$. But $J^{\sv}(\tone_{\infty})=1$ which implies that $A'=1$.
\end{proof}
The previous result can also be deduced from the fact that complex conjugation, and hence the elements $\mathbf{s}$ and $\sv$ in the proof of theorem $ \ref{thmExistbphi}$,  are  involutions. This  implies a cumbersome identity involving  $b^{\sv}$,  $\phi^{\sv}$ and $\sv$ which we chose to omit.

\section{Definition of a class of real-analytic modular forms}

Having defined a modular-equivariant function $J^{\eqv}$, we can extract real-analytic modular forms from its coefficients, which in turn generate the algebra $\MI^E$.

\subsection{Coefficients of $J^{\eqv}$}
For every $\tau \in \HH$, 
we view $J^{\eqv}(\tau) \in \Ue(\C)$ as a homomorphism $\Or(\Ue) \rightarrow \C$. 
For every $\SL_2$-equivariant  $\Q$-linear map
\begin{equation} \label{coefffunction} c:  (V^{dR}_{2n})^{\vee} \To \Or(\Ue)
\end{equation}
consider the composite (all tensor products  are over $\Q$)
$$ \HH\overset{J^{\eqv}}{\To} \mathrm{Hom} (\Or(\Ue),\C) \overset{c}{\To} \mathrm{Hom}((V^{dR}_{2n})^{\vee}, \C) \cong V_{2n} \otimes \C $$
which we call the `coefficient of $c$ in $J^{\eqv}$'.  The last  isomorphism  is induced by  the comparison $V^{dR}_{2n} \otimes     \C \cong V_{2n} \otimes  \C$ (\ref{VBdRcomp}). The coefficient of $c$ is a real analytic map 
$$c(J^{\eqv}): \HH \To V_{2n}\otimes \C\ .$$
Since $J^{\eqv}$ is equivariant, it follows that for all $\gamma \in \SL_2(\Z)$: 
$$c(J^{\eqv})(\gamma \tau) \big|_{\gamma} =  c(J^{\eqv})\ .$$

\subsection{Equivariant sections}The following result  was proved in \cite{ZagFest}, \S7.
\begin{prop}  \label{propmodularformsfromsections}  A real-analytic section  $f: \HH \rightarrow V_{2n}\otimes  \C$
$$f(\tau) = \sum_{r+s=2n}  f^{r,s}(\tau) X^r Y^s$$
can be uniquely written in the form 
\begin{equation} \label{fsection2} f(\tau)= \sum_{r+s=2n} f_{r,s} (\tau)(X-\tau Y)^r (X- \overline{\tau} Y)^s \ , 
\end{equation}
where  $  f_{r,s}(\tau) : \HH \rightarrow \C$ are real analytic. 
Then $f$ is equivariant if and only if  the coefficients $f_{r,s}(\tau)$ are modular of weights $(r,s)$  for all $r,s$.
If  in addition the $f^{r,s}(\tau)$ admit  expansions in the ring 
$ \C[[q, \overline{q}]][\tau, \overline{\tau} ]$
then the $f_{r,s} \in P^{-(r+s)} \mathcal{M}_{r,s}$.
\end{prop} 
The proposition follows from applying the invertible change of variables
\begin{eqnarray}  \label{XYinverse}
X &  \mapsto  &   \frac{ \tau}{\tau- \overline{\tau}}  (X-\overline{\tau}Y) -   \frac{  \overline{\tau}}{\tau- \overline{\tau}}  (X- \tau Y)     \\
 Y &  \mapsto  &   \frac{1}{\tau- \overline{\tau}}  (X-\overline{\tau}Y) -   \frac{1}{\tau-\overline{\tau}}  (X- \tau Y)  \ .\nonumber 
 \end{eqnarray} 
 In the de Rham basis of $V^{dR}_{2n}$, this corresponds to 
 \begin{eqnarray}  \label{dRXYinverse}
\Xv &  \mapsto  &   \frac{ \log(q)}{2   \LL}  (\Xv +  \log(\overline{q})\Yv) +  \frac{ \log (\overline{q})}{2 \LL}  (\Xv-  \log(q) \Yv)     \\
 \Yv &  \mapsto  &   \frac{1}{2 \LL}  (\Xv +  \log(\overline{q})\Yv) -   \frac{1}{2 \LL}  (\Xv-  \log(q) \Yv)  \nonumber 
 \end{eqnarray}

\subsection{Definition of the space $\MI^E$} \label{sect: ModComp}
Define   the \emph{modular components}
$c_{r,s} (J^{\eqv})$
of a coefficient function $c$ to be the functions defined in the manner of proposition \ref{propmodularformsfromsections}.  
 They are the unique functions satisfying 
$$\sum_{r+s=2n} c_{r,s}  (X-z Y)^r (X-\overline{z} Y)^s  =  c(J^{\eqv})(X,Y) \ .$$
We shall show in theorem \ref{thm: SVMZVcoeff} below  that the $c_{r,s}(J^{\eqv})$ lie in $\mathcal{M}_{r,s}$. 
\begin{defn}
Let $\MI^E\subset \mathcal{M}$ denote the  $\ZZ^{\sv}$-module generated by all  modular components $c_{r,s}$ of all coefficients of $J^{\eqv}$.   \end{defn}

The space $\MI^E$ is well-defined, i.e., independent of the choices  $b^{\sv}, \phi^{\sv}$ in theorem \ref{thmExistbphi}.
Multiplying  $J^{\eqv}$ on the right by an element $a \in (\Ue)^{\SL_2}(\ZZ^{\sv})$ amounts to modifying the  coefficients $c$ by a linear transformation with coefficients in  $\ZZ^{\sv}$.

\section{First properties of $\MI^E$} \label{sect: FirstProp}

Since the space $\MI^E$ is generated by the coefficients (\ref{coefffunction}), it inherits a number of structures from the affine ring
 $\Or(\Ue)$.  Some first examples are discussed here.
 
\subsection{Multiplicative structure}

\begin{prop}  \label{prop: algebra} The space $\MI^E[\LL^{\pm}]$ is closed under multiplication.
\end{prop}

\begin{proof}
Since $\Or(\Ue)$ is a commutative ring,   any two $\SL_2$-equivariant maps 
$$c : \dV^{dR}_{2m}\To \Or(\Ue) \quad  \hbox{ and } \quad c' : \dV^{dR}_{2n} \To \Or(\Ue)$$
can be multiplied together via 
$$   \dV^{dR}_{2m} \otimes    \dV^{dR}_{2n} \,  \overset{c\otimes c'}{\To} \, \Or(\Ue) \otimes \Or(\Ue) \,  \To  \, \Or(\Ue)$$
where the second map is  the multiplication on $\Or(\Ue)$. By composing with the inclusion of an isotypical  factor in an $\SL_2$-equivariant decomposition
$$\dV^{dR}_{2m+2n-2k} \quad  \hookrightarrow  \quad \dV^{dR}_{2m+2n} \oplus \ldots \oplus \dV^{dR}_{2|m-n|} \quad  \cong  \quad \dV^{dR}_{2m} \otimes \dV^{dR}_{2n}  $$
of the left hand side, we deduce  that the product of coefficients $c,c'$ can be decomposed as a linear combination of coefficients of the form
$$c_k'': \dV^{dR}_{2m+2n-2k}  \To \Or(\Ue)$$
for  $0\leq k \leq \min\{m,n\}$.  Since, furthermore, the transformation $(\ref{dRXYinverse})$  is linear with coefficients in  $\Q[\LL^{\pm}, \log q  , \log \overline{q}]$, it follows that every product  $c_{r,s} . c'_{r',s'}$ of modular components  (\S \ref{sect: ModComp})
can in turn be expressed as a linear combination, with coefficients in $\Q[\LL^{\pm} , \log q , \log \overline{q}]$, of $c''_{r+r',s+s'}$. In fact, the coefficients must lie in $\Q[\LL^{\pm}]$,  because any modular form
satisfying (\ref{gammaact})  is translation-invariant, and so  
the coefficients in this formula   must lie in the  $T$-invariant subspace of  $\Q[\LL^{\pm}, \log q  , \log \overline{q}]$, which is exactly  $\Q[\LL^{\pm}]$  
(see \cite{ZagFest}, lemma 2.2). 
\end{proof} 

In particular, for any $r,s,r',s'\geq 0$,
\begin{equation}  \label{MIEproduct}
\MI^E_{r,s} \times \MI^E_{r',s'} \quad \subset \quad  \MI^E[\LL^{\pm}]_{r+r',s+s'}\ .
\end{equation} 
Note that $\MI^E$ itself is not an algebra. 
\subsection{Complex conjugation}
The space $\MI^E$ is stable under the action of complex conjugation. This follows from the definition and lemma \ref{lemComplexConj}, since the coefficients of $b^{\sv}$ and $\phi^{\sv}$ lie in the ring $\ZZ^{\sv} \subset \MI^E$. 

\subsection{Length filtration}
The affine ring $\Or(\Ue)$ is equipped with an increasing length filtration $L$, which is dual to the lower central series on $\ue$. 

\begin{defn}This induces an increasing filtration 
$$\MI^E_{\ell} = \langle c_{r,s} \ : \      \mathrm{Im}\, (c) \subseteq  L_{\ell}  \Or(\Ue) \rangle_{\ZZ^{\sv}}\ .$$
\end{defn}
Elements in $\MI^E_{\ell}$ are generated by  the real and imaginary parts of iterated integrals of Eisenstein series of length less than or equal to $\ell$. 
The length filtration is compatible with the multiplicative structure defined above, since $L_{u} L_v \subset L_{u+v}$. 

\subsection{$M$-grading}
The affine ring $\Or(\Ue)$ is non-negatively graded  by the $M$-degree, and furthermore, the $M$-degree is compatible with the action of $\SL_2$. In particular, 
the  generators (\ref{XYderivs}) of $\ssl_2$ have $M$-degrees $-1,0,+1$. 

Likewise, the space $V^{dR}_{2n}$ inherits an $M$-grading  with degrees $0, -1, \ldots, -2n$ from $(\ref{degsXY})$. Its dual  $\dV^{dR}_{2n}$ has $M$-degrees $0,1,\ldots, 2n$. 

\begin{lem} Every non-zero coefficient function 
$$c: \dV^{dR}_{2n} \To \Or(\Ue)$$
is  a map of $M$-graded $\SL_2$-modules of degree $m$, for some  $m \geq 0$. 
It follows that its modular components $c_{r,s}$  inherit  an $M$-grading from $\Or(\Ue)$ satisfying
$$\deg_M c_{r,s} = \deg_M c =m  \qquad \hbox{ for all } r, s  \geq 0  \  . $$ 
\end{lem} 
\begin{proof} The fact that  a coefficient function respects the $M$-grading follows from the compatibility between the $M$-grading and $\SL_2$-action:
any coefficient function is generated, under the action of $\SL_2$, by its image on a lowest weight vector. 
Since $\Or(\Ue)$ is concentrated in non-negative $M$-degrees,  and $\V^{dR}_{2n}$ has a component in $M$-degree $0$, it follows that 
any coefficient function must have degree $m\geq 0$. 
 
 Finally,   we define
  $\deg_M 2 \pi i \tau = \deg_M 2 \pi i \overline{\tau}=1$. This  implies that 
 \begin{equation} \label{degMLL} \deg_M \LL = 1\ .
\end{equation} 
 With these definitions\footnote{Note that the $M$-degree of $\LL$ is compatible with the Hodge-theoretic weight of the function $2\, \log|q|$, which is the single-valued period
of a family of Kummer extensions, and has Hodge-theoretic weight $2$.  Recall that $M$ is normalised to be one half of the Hodge theoretic weight. The motivic Lefschetz period  which is associated to $2\pi i$ also has Hodge-theoretic weight $2$, so these definitions are forced upon us from the fact that $\tau$ and $\overline{\tau}$ have $M$-degree zero. },  the terms in  $(\ref{dRXYinverse})$ are homogeneous  of degree zero, since 
 $\deg_M X=0$ and $\deg_M Y=-1$. Therefore $\deg_M c_{r,s} = \deg_M c^{r,s}= \deg_M c$. 
\end{proof}

\begin{defn}
The $M$-grading defines an increasing filtration  in the usual manner, and 
  extends to an $M$-filtration on $\MI^E$, where the $M$-filtration on $\ZZ^{\sv}$ is half the usual Hodge-theoretic weight filtration (the multiple zeta value weight). 
 \end{defn}
 
 The  $M$-filtration is well-defined  (independent of the choice of representatives $b^{\sv}, \phi^{\sv}$ in theorem \ref{thmExistbphi}) since it is induced by the $M$-filtration on $\Or(\Ue)$. Another way to say this is that the coproduct on $\Or(\Ue)$ dual to multiplication in $\Ue$ is compatible with the $M$-filtration, so twisting by an element $a$, as in theorem \ref{thmExistbphi}, only modifies the coefficient functions by elements of lower $M$-weight. 
 
  \begin{rem} 
 Conjecturally, the weight filtration is a grading on $\ZZ^{\sv}$.   One way to exploit this is to work with motivic periods instead of real numbers, since  motivic single-valued multiple zeta values are indeed weight-graded. This amounts to replacing $b^{\sv}, \phi^{\sv}$ with their `motivic' versions, defined in 
 \cite{MMV}.  Then the weight grading on the motivic periods is precisely dual to the $M$-grading on $\Or(\Ue)$. 
 It remains to verify that this grading is not disturbed when extracting coefficients and  modular components \S\ref{sect: ModComp}.
 To see this, note that the  comparison isomorphism $(\ref{VBdRcomp})$ must be replaced with the universal comparison isomorphism in which one replaces $2\pi i$ by its motivic version $(2 \pi i)^{\mm}$, which  is  Tate, and $M$-graded of degree $1$. 
 The motivic version of the change of variables (\ref{dRXYinverse}), in which one replaces $2 \pi i$ by $(2 \pi i)^{\mm}$,  is therefore homogeneous of total degree $0$,
 which implies that $\deg_M c_{r,s} = \deg_M c$. 
 It follows that  the $M$-grading on $\Or(\Ue)$ induces an  $M$-grading on all coefficient functions.    This leads to a class of `motivic' modular forms which are formal expansions  \eqref{intro: fqcoeffs} whose coefficients    are motivic periods (in this case, motivic  single-valued multiple zeta values). Their image under the period homomorphism are genuine modular forms in $\mathcal{M}$.   This will be discussed elsewhere. 
Note that  the $M$-grading can be used to determine the  powers of $\LL$ in the right-hand side of (\ref{MIEproduct}).
 \end{rem}
 
 Since  the generators of $\ue$ have strictly negative $M$-degrees,  and since $M_0 \ZZ^{\sv} \cong \Q$ we deduce that 
$M_0 \MI^E = \Q$.

\subsection{Expansions and coefficients}
\begin{thm}  \label{thm: SVMZVcoeff} Every element of $\MI^E$ admits an expansion in the ring
$$\ZZ^{\sv} [[q, \overline{q}]][\LL^{\pm}]\ ,$$
i.e., its coefficients are single-valued multiple zeta values.  An element of total modular  weight $w$ has poles in $\LL$ of order at most $w$. 
\end{thm} 
\begin{proof} By theorem \ref{thmExistbphi}, the coefficients of $b^{\sv}, \phi^{\sv}$ are in the ring $\ZZ^{\sv}$. 
It follows from the definition of $J^{\eqv}$ that its coefficients are $\ZZ^{\sv}$-linear combinations of real and imaginary parts of coefficients of $J$. 
By proposition \ref{propIE}, the latter are iterated integrals with   expansions in  the ring $\Q [[q]][\log q]$.
Finally, the change of variables $\ref{dRXYinverse}$ is defined over $\Q[\LL^{\pm} , \log q, \log \overline{q}]$. This proves that  every element of
$\MI^E$ admits an expansion in 
$$\ZZ^{\sv} [[q, \overline{q}]][\LL^{\pm} ,\log q, \log \overline{q} ]\ .$$
A modular form is  translation invariant, so its expansion lies in the subring of $T$-invariants, which is precisely
$\ZZ^{\sv} [[q, \overline{q}]][\LL^{\pm}]$ (see \cite{ZagFest}, lemma 2.2).  The bound on the order of poles in $\LL$ follows from the definition 
of the coefficients and  $(\ref{dRXYinverse})$.
\end{proof}

\subsection{Finiteness}

\begin{thm}  \label{prop: finite}
The subspace  of elements of  $\MI^E$ of total  modular weight $w$  and $M$-degree $\leq m$ is finite-dimensional for every $m,w$. 
\end{thm} 

\begin{proof}  It is enough to show that the dimension of each $\SL_2$-isotypical component of $\gr^M_m \Or(\Ue)$, or equivalently $\gr^M_{-m} \ue$, is finite-dimensional. Since the latter is a quotient of $\uu^{dR}_E$, it is enough to prove it for  the free Lie algebra on $\e_{2n+2} V^{dR}_{2n}$. Note in passing that since $\gr_M \e_{2n+2}=-1$, 
the space $\gr_{-m}^M \uu^{dR}_E$ only involves Lie words of bounded length in the $\e_{2n+2} V^{dR}_{2n}$. 
Now  the  inclusion of an isotypical component  
$$V^{dR}_{2\ell} \To  V^{dR}_{2n_1} \otimes \ldots \otimes V^{dR}_{2n_r}$$
has $M$-degree $\ell -(n_1+ \ldots +n_r )$.  This follows from the fact that the map $\delta_{dR}^k$ defined in (\ref{deltakdef}) has $M$-degree $k$, so its dual has $M$-degree $-k$, where $k=n_1+ \ldots +n_r-\ell$. The integer $\ell$ is constrained by the modular weight:  $\ell=w$. 
The statement follows from the fact that the number of strictly positive integers $n_1,\ldots, n_r$ whose sum is bounded above is finite, and that the subring of  $\ZZ^{\sv}$ of bounded weight is finite-dimensional. 
\end{proof}

\subsection{Compatibilities}
The above structures are all mutually compatible: for example, the length filtration and $M$-filtration are compatible with the algebra structure.

\section{Differential structure of $\MI^E$}

The ordinary differential equations satisfied by $J^{\eqv}$ are equivalent to  a system of differential equations involving the operators $\partial$ and $\overline{\partial}$ 
 satisfied by its modular components.  These in turn give rise to inhomogeneous Laplace eigenvalue equations.

\subsection{Vector-valued differential equations}
\begin{prop} \label{propdF=A}  Let $F,A,B: \mathcal{H} \rightarrow V_{2n}\otimes \C$ be real analytic. Then the equation 
\begin{equation} \label{dF=A} {\partial F \over \partial z} = {2 \pi i  \over 2}  A(z)
\end{equation} 
is equivalent to the following system of equations: 
\begin{eqnarray}
 \partial F_{2n,0}  & = &   \LL  A_{2n,0}   \\
 \partial  F_{r,s} - (r+1) F_{r+1, s-1}  &= & \LL  A_{r,s} \qquad \hbox{ if } \quad s\geq 1 \ ,  \nonumber 
\end{eqnarray} 
for all $r+s = 2n$, and $r,s \geq 0$.
In a similar manner, 
\begin{equation} \label{dbarF=B} {\partial F \over \partial \overline{z}} = {2 \pi i  \over 2}  B(z)
\end{equation} 
is equivalent to the following system of equations:
\begin{eqnarray}
 \overline{\partial}  F_{0,2n}  & = &   \LL  B_{0,2n}   \\
 \overline{\partial} F_{r,s} - (s+1) F_{r-1, s+1}  &= & \LL  B_{r,s} \qquad \hbox{ if } \quad r\geq 1 \ .  \nonumber 
\end{eqnarray} 
\end{prop} 
The proof is a straightforward computation (\cite{ZagFest} \S7).

\begin{lem} \label{lemapplydeltak}  Suppose that $A: \HH \rightarrow V_{2n} \otimes\C$ and set
$$F= {\delta_{dR}^k \over (k!)^2} \Big(  (X-zY)^{2m} \otimes A \Big)\ . $$
Then  $F: \HH \rightarrow V_{2m+2n-2k}\otimes \C$ vanishes if $k>2n $ or $k>2m$, and its components  satisfy
\begin{equation}
F_{ r,s }  =    (2\LL)^k \,     \binom{2m}{k}\binom{s+k}{k} A_{r-2m+k, s+k}   
  \end{equation}
where we set $A_{p,q}=0$ for $p<0$ or $q<0$. Therefore  $F_{r,s}$ vanishes   if $r<2m-k$, or equivalently, $s+k>2n$. 
\end{lem} 
\begin{proof}
The statement with  $\delta_{dR}$  replaced by $\delta$ is given in  \cite{ZagFest} \S7. Use $\delta_{dR} = 2\pi i \delta$. 
\end{proof}

\subsection{Differential structure of $\MI^E$}
Recall that the graded $\Q$-vector space generated by Eisenstein series is denoted by
$$E = \bigoplus_{n\geq 1} \GE_{2n+2} \Q \ .$$
It is  placed in $M$-degree zero.

\begin{thm}  \label{thm: moddiff} Every modular component $c_{r,s}$ of a coefficient function  $c$ satisfies
\begin{equation} \label{holcrs} \partial c_{r,s} - (r+1) c_{r+1, s-1} \quad \in \quad E[\LL] \times \MI^E_{k-1} 
\end{equation} 
where we define $c_{r,s}$ to be zero if  $r$ or $s$ is negative. Similarly,
\begin{equation} \label{antiholcrs} \overline{\partial} c_{r,s} - (s+1) c_{r-1, s+1}   \quad \in \quad \overline{E}[\LL] \times \MI^E_{k-1}\ .
\end{equation}
\end{thm}

\begin{proof}  Let $c:\dV^{dR}_{2n} \rightarrow \Or(\Ue)$ be a coefficient. Embed  $\Or(\Ue) \subset \Or(\U^{dR}_E)$ via $\mu^*$, where $\mu$ is the monodromy homomorphism.  We shall also denote the induced  $\SL_2$-equivariant map $ \mu^* c:  \dV^{dR}_{2n} \rightarrow \Or(\U^{dR}_E)$  by $c$ and view a coefficient function 
 as taking values in $\Or(\Ue)$ or $\Or(\U^{dR}_E)$ as appropriate. If we
 choose any splitting $\Or(\U^{dR}_E) \rightarrow \Or(\Ue)$ as $\SL_2$-modules, we can write  $J^{\eqv}$, via definition  \ref{defnJeqv} in the form 
$$J^{\eqv}(\tau) = \mu  \big(  I^E   \overline{ K}  \big) $$
for some antiholomorphic function  $\overline{K} : \HH \rightarrow \Ue(\C)$ which is viewed as function taking values in the scheme  $\U^{dR}_E(\C)$  via our choice of splitting. Therefore 
$${ \partial \over \partial \tau} c(J^{\eqv}) = -   c\left( \mu(\widetilde{\Omega}^E I^E  \overline{K}) \right)  \ , $$
where 
$$ \widetilde{\Omega}^E =  2 \pi i \sum_{n\geq 1}  \e_{2n+2}\,  \G_{2n+2} (\tau) (X- \tau Y)^{2n}\ ,$$
which follows from the differential equation $d I^E = - \Omega^E I^E$ and (\ref{OmegaEdef}). The coproduct 
$$\Delta: \Or(\U^{dR}_E) \To \Or(\U^{dR}_E) \otimes \Or(\U^{dR}_E)$$ 
dual to the multiplication law in $\U^{dR}_E$ is $\SL_2$-equivariant.  In particular, it  induces a coproduct on coefficient functions 
$c: \dV^{dR}_{2n} \rightarrow \Or(\U^{dR}_E)$ in the following the form (using a variant of Sweedler's notation)
$$\Delta \,c = \sum_{k\geq 0}  (c' \otimes c'') (\delta_{dR}^k)^{\vee}\ ,$$
where 
$(\delta_{dR}^k)^{\vee}:  \check{V}^{dR}_{2n}  \subset \check{V}^{dR} \rightarrow \check{V}^{dR} \otimes  \check{V}^{dR}$
is the dual of $\delta_{dR}^k: V^{dR} \otimes V^{dR} \rightarrow V^{dR}$, and $V^{dR}= \bigoplus_{n\geq0} V^{dR}_n = \Q[\Xv,\Yv]$. 
Let $\langle, \rangle$ denote the pairing between the points of $\Ue$ and $\Or(\Ue)$.  We can rewrite our differential equation in the form
$${ \partial \over \partial \tau} c(J^{\eqv}) = -   \langle c,   m (\mu \otimes \mu) (\widetilde{\Omega}^E \otimes I^E \overline{K})\rangle  $$
where $m$ denotes multiplication, and $\widetilde{\Omega}^E$ is viewed as a section of $L_1\Or(\U^{dR}_E)$. 
By the duality between coproducts and multiplication, this equals
\begin{eqnarray}
{ \partial \over \partial \tau} c(J^{\eqv})  &=  & -  \langle \Delta c, (\mu \otimes \mu)(\widetilde{\Omega}^E \otimes I^E \overline{K} )\rangle \nonumber   \\
 &=  & -  \sum_{k\geq 0} \langle   (c' \otimes c'') (\delta_{dR}^k)^{\vee}     ,  \widetilde{\Omega}^E \otimes  J^{\eqv} \rangle \nonumber \\
 &=  & -  \sum_{k\geq 0}   \delta_{dR}^k  \left( (c' \otimes c'') (\widetilde{\Omega}^E \otimes J^{\eqv}) \right) \nonumber 
\end{eqnarray}  
where in the second line, we view  $c''$ as a coefficient function on $\Or(\Ue)$ and $c'$ as a coefficient function on $\Or(\U^{dR}_E)$. 
Restricted to $L_1\Or(\U^{dR}_E)$,  the map $c'$ is a linear combination of the maps which send one $e_{2r+2}$ to $1$ and all other $e_{2r'+2}$ to $0$. 
We conclude that
$${ \partial \over \partial \tau} c(J^{\eqv}) = 2\pi i  A$$
where $A: \HH \rightarrow V_{2n}\otimes \C$ is a  $\Q$-linear combination of  terms of the form
$$ \delta^k_{dR}  \Big( \GE_{2m+2}(\tau)(X-\tau Y)^{2m} \otimes c'' (J^{\eqv}) \Big)$$
where $c''$ is a coefficient function of strictly smaller length than $c$, since the coproduct $\Delta$ is compatible with the length filtration.
The first part of theorem follows on 
applying  proposition $\ref{propdF=A}$ and lemma \ref{lemapplydeltak}. The second part follows from the first using the fact that  $\MI^E$ is  stable
under complex conjugation.
 \end{proof}

Since $\MI^E$ is generated by coefficient functions we deduce the
\begin{cor}
The space $\MI^E$ has the following differential structure:
\begin{eqnarray} 
\partial  \big( \MI_k^E\big)    &\subset &    \MI^E_{k}  \quad +  \quad E[\LL]  \times \MI_{k-1}^E    \nonumber \\
\overline{\partial}  \big( \MI_k^E \big)  &\subset &    \MI^E_{k} \quad + \quad   \overline{E}[\LL]  \times \MI_{k-1}^E    \ . \nonumber 
\end{eqnarray} 
The operators $\partial, \overline{\partial}$ respect the $M$-filtration, i.e., $\deg_M \partial= \deg_M \overline{\partial}=0$. 
\end{cor}
\begin{proof} The first part is immediate from the previous theorem. The statement about the $M$-filtration follows since the coproduct $\Delta$
respects the $M$-grading on $\Or(\U^{dR}_E)$, the Eisenstein series $\GE_{2m+2}$ lie in $M$-degree zero, and the fact that in lemma \ref{lemapplydeltak}, the $M$-degree of the powers of $\LL$  match the $M$-degree of $\delta_{dR}^k$.
\end{proof}

By proposition \ref{propModularKernel}, an element  $\xi \in \mathcal{M}_{r,s}$ is uniquely determined by $\partial \xi$ and $\overline{\partial} \xi$, up to a possible multiple of $\LL^{-r}$ in the case $r=s$.  When $\xi \in \MI^E$, this constant is an element of $\ZZ^{\sv}$, whose $M$-weight can be determined from the $M$-grading. 

\begin{rem}  \label{remsplitRHS} By the independence of iterated integrals (corollary \ref{corLIoverModforms}), the sums on  the right-hand side in the previous corollary are direct, and so we may write
$$\partial \big( \MI_k^E\big)   \quad   \subset \quad     \MI^E_{k}  \quad    \oplus  \quad \big(  E[\LL]  \times \MI_{k-1}^E \big)   $$
and similarly for $\overline{\partial}$. This is because $E$ does not contain any  constant functions.
\end{rem}

\subsection{Reconstruction of vector-valued modular forms}\label{remsl2onMIE}

Given an element $f_{m,n} \in \MI^E_k$ of modular weights $m,n \geq 0$, we can use the splitting of remark \ref{remsplitRHS}  to define functions $f_{r,s} \in \MI^E_k$  for all $m+n= r+s$, and $r,s \geq 0$ via
$$\partial f_{r,s} = (r+1) f_{r+1, s-1}  \pmod{ E[\LL] \times \MI^E_{k-1}}$$
whenever  $s>1$ (since $r+1\neq 0$)  and via 
$$\overline{\partial} f_{r,s} = (s+1) f_{r-1, s+1}  \pmod{\overline{E}[\LL] \times \MI^E_{k-1}}$$
whenever $r>1$ (since $s+1 \neq 0$).  That these are equations are consistent  follows from $[\partial , \overline{\partial}] = h$. The function 
$$F(\tau) = \sum_{r,s} f_{r,s}(X-\tau Y)^r (X- \overline{\tau} Y)^s$$
is then a vector-valued  modular form, and can be   reconstructed from any one of its individual modular components. 

\subsection{Laplace operator for vector-valued functions}
The following lemma explains the  existence of Laplace eigenvalue equations in a general setting. 
\begin{lem} Let $F: \HH \rightarrow V_{2n} \otimes \C$ be real analytic satisfying the equation 
$$ dF =  {(2 \pi i )\over 2}    \big( A dz + B d \overline{z} \big) $$ 
for some $A, B: \HH \rightarrow V_{2n} \otimes \C$.  Then 
\begin{eqnarray} \big( \Delta +  r+s \big) \, F_{r,s}   & = &   \LL\, \big(  \overline{\partial} A_{r,s} +   (r+1)   B_{r+1,s-1} \big)  \label{LaplaceFrs} \\ 
& = &   \LL\,  \big(      \partial B_{r,s} + (s+1)  A_{r-1,s+1} \big)    \nonumber 
\end{eqnarray} 
where $A_{r,s}, B_{r,s}$ are understood to be zero if any subscript $r$ or $s$ is negative.
\end{lem} 

\begin{proof} The differential equation $d^2 F = 0$ implies that  ${\partial B \over \partial z} - {\partial A \over \partial \overline{z}}=0$. By proposition $\ref{propdF=A}$ this is equivalent to the equations
$$\partial B_{r,s}  - (r+1) B_{r+1,s-1} =      \overline{\partial} A_{r,s} - (s+1) A_{r-1,s+1}   $$
for all $r+s = 2n$, 
This shows that the two expressions $(\ref{LaplaceFrs})$ are equivalent. 
  Again by proposition  $\ref{propdF=A}$  and the relation $[\overline{\partial}, \LL]=0$, we verify that 
\begin{eqnarray} \overline{\partial} \partial F_{r,s }   &= &  (r+1)  \overline{\partial} F_{r+1, s-1} + \LL \,  \overline{\partial} A_{r,s} \nonumber  \\ 
 & = & (r+1)s  F_{r,s}  + \LL \,   \overline{\partial} A_{r,s}  + (r+1) \LL \, B_{ r+1 ,s-1 }\ . \nonumber 
\end{eqnarray} 
The statement follows from the definition of the Laplacian (\ref{LaplaceDef}).
\end{proof}

\subsection{Laplace operator structure for $\MI^E$}

\begin{cor} \label{cor: Laplace}
Every element $F \in \MI_k^E$ of modular weights $(r,s)$ satisfies an inhomogeneous Laplace equation of the following form:
$$ (\Delta + r+s )\, F  \quad \in \quad  (E+ \overline{E})[\LL] \times \MI^E_{k-1} + E \overline{E} [\LL]  \times \MI^E_{k-2} \ ,$$
where the  eigenvalue is minus the total modular weight.  
\end{cor}
\begin{proof}
This follows from the previous lemma, and equations (\ref{holcrs}) and (\ref{antiholcrs}) (or by direct application of the definition of the Laplace operator, using these same two equations, and the Leibniz rule).
\end{proof} 
The sum in the right-hand side is direct, by corollary \ref{corLIoverModforms}. It could be written 
$$(E[\LL]  \times \MI^E_{k-1}) \quad \oplus \quad ( \overline{E}[\LL] \times \MI^E_{k-1})  \quad \oplus \quad (E \overline{E} [\LL]  \times \MI^E_{k-2})\ . $$

\section{Algebraic structure of $\MI^E$}
In this section we delve more deeply into the algebraic structure of $\MI^E$.
Since the space $\MI^E$ is generated from the coefficients of $\Or(\Ue)$, its structure is closely related to that of the geometric Lie algebra $\ue$. 
Although the precise structure of the latter is not completely known,   we can use the relationship with $\MI^E$ to transfer information back and forth between modular forms and derivations in $\ue$.

\subsection{Algebraic structure and dimensions}   Let us denote the  subspace of  lowest weight vectors  for $\SL_2$ in $\Or(\Ue)$ by 
$$\lw (\Or(\Ue)) \quad  \subset  \quad \Or(\Ue)\ .$$
 It is a subalgebra of $\Or(\Ue)$,  filtered by length, and graded with respect to $M$. 

 On the other hand, consider
$$\mathcal{M}_{\bullet, 0} = \bigoplus_{n\geq 0} \mathcal{M}_{n,0}$$
which defines a subalgebra of $\mathcal{M}$.  It consists of functions which transform like  classical modular forms (with no anti-holomorphic factor of automorphy).  Let us define 
$$\lw( \MI^{E})  =  \MI^E \cap \mathcal{M}_{\bullet, 0}\ ,$$
and call the elements lowest weight vectors. 
Although the Lie algebra $\ssl_2$ does not  act on $\MI^E$ \emph{per se}, the terminology is justified since images of  elements in $\lw (\MI^E)$ behave like lowest weight vectors 
in the context of remark \ref{remsl2onMIE}.

\begin{thm} \label{thmMIEandOU} The subspace 
$\lw (\MI^E) \subset \MI^E$
is closed under multiplication.

There is a canonical  $\ZZ^{\sv}$-linear isomorphism of algebras
\begin{equation} \label{OuetoMIE}  \gr^L_{\bullet}\,  \lw (\Or(\Ue)) \otimes  \ZZ^{\sv} \overset{\sim}{\To}  \gr^L_{\bullet} \,\lw (\MI^{E} )\end{equation} 
It respects the  $M$-filtration on both sides of the isomorphism. 
\end{thm} 

\begin{proof}
 Fix a choice of  elements $(b^{\sv},\phi^{\sv})$ as in theorem \ref{thmExistbphi} and define $J^{\eqv}(\tau)$ according to  definition \ref{defnJeqv}. A non-trivial lowest weight vector $v\in \mathrm{lw}(\Or(\Ue))$ of $\mathrm{SL}_2$-weight $n$ generates, under the action of $\SL_2$, an irreducible $\SL_2$-submodule
$$c_v: \dV^{dR}_{n}\quad \subset \quad \Or(\Ue)\ ,$$
and furthermore, every irreducible $\SL_2$-submodule  arises in this way. 
Taking the coefficient $c_v(J^{\eqv})$ and extracting the term $(c_v)_{n,0}$ in the manner of proposition $\ref{propmodularformsfromsections}$ defines a modular form in
$ \MI^E $ of modular weights  $(n,0)$. It is   given explicitly by
\begin{equation} \label{cvmodexpl}  \chi(v) \ =  \ \LL^{-2n}  c_v(J^{\eqv}) \big|_{X=  \pi i  \overline{z}, \, Y= \pi i } \quad \in \quad  \MI^E \cap \mathcal{M}_{n,0}\  .
\end{equation} 
 This extends to a $\ZZ^{\sv}$-linear
 map
$$\chi\ : \  \lw (\Or(\Ue)) \otimes \ZZ^{\sv} \To  \lw (\MI^{E} )$$
which respects the $M$ and $L$ filtrations.  It depends on the choice of $(b^{\sv},\phi^{\sv})$. 
By definition of $\MI^E$, every  modular form of weights $(n, 0)$ arises in this way, and $\chi$  is surjective.
To prove injectivity, it is enough to show that the associated graded of $\chi$ with respect to the length filtration is injective. For this, consider lowest weight vectors $v_1,\ldots, v_n$  in $L_k \Or(\Ue)$ which are linearly independent  in $\gr^L_k \Or(\Ue)$. They define linearly disjoint $\SL_2$-submodules  $c_{v_1},\ldots, c_{v_n}$ of $\Or(\Ue)$.
The dual of the monodromy homomorphism  (\ref{defnuEtoue}) defines an embedding
$$\Or(\Ue) \hookrightarrow   \Or(\U^{dR}_E) = T^c \big(\bigoplus_{n \geq 1} \EE_{2n+2} \dV^{dR}_{2n} \big) $$
By the linear independence of iterated integrals  (corollary \ref{corLIoverModforms}), 
the corresponding modular forms  $\chi(v_1), \ldots, \chi(v_n)$ are linearly independent modulo iterated integrals of length $\leq k-1$, since by 
the remarks following theorem  \ref{thmJeqv}, their leading terms are real and imaginary parts of iterated integrals of independent Eisenstein series. 
This proves injectivity. 

  We next show that $\chi$ is a homomorphism. Let $v_1, v_2 \in \lw (\Or(\Ue))$ of $\mathrm{SL}_2$-weights $n_1, n_2$.  Then $c_{v_1v_2}$ is   defined via the commuting diagram
$$  
\begin{array}{ccc}
 \dV^{dR}_{n_1} \otimes   \dV^{dR}_{n_2}  &  \longleftarrow   &  \dV^{dR}_{n_1+n_2} \\
  \downarrow_{c_{v_1}} \qquad \downarrow_{c_{v_2}} &   &  \qquad \downarrow_{c_{v_1v_2}}  \\
  \Or(\Ue) \otimes  \Or(\Ue)  & \overset{m}{\To}  & \Or(\Ue) \\
  \downarrow \quad \downarrow &   &  \downarrow  \\
 \C \otimes  \C & \overset{m}{\To}  & \C    
\end{array}
$$
  where   the map along the top is dual to $\delta_{dR}^0 : V^{dR}_{n_1}\otimes V^{dR}_{n_2} \rightarrow V^{dR}_{n_1+n_2}$,  
  the vertical maps in the bottom square are given by the homomorphism $J^{\eqv}(\tau) : \Or(\Ue) \rightarrow \C$,  and $m$ denotes multiplication. 
  The reason that the top square commutes is because the corresponding square with $\Or(\Ue)$ replaced with $\Or(\U^{dR}_E)$ commutes, and the monodromy map $\mu^*: \Or(\Ue)\subset \Or(\U^{dR}_E)$ is $\mathrm{SL}_2$ equivariant and respects multiplication. 
  It follows that 
    $$c_{v_1v_2} (J^{\eqv}) = \delta_{dR}^0( c_{v_1} (J^{\eqv}) c_{v_2}(J^{\eqv}))\ .$$
  From the definition of $\chi$, we obtain  $\chi(v_1v_2) = \chi(v_1)\chi(v_2)$. 
  Since $\chi$ is an isomorphism, this also implies that $\lw (\MI^E) \subset \MI^E$ is stable under multiplication.

Finally, the isomorphism (\ref{OuetoMIE}) is obtained by replacing $\chi$ with its associated graded for the length filtration.
It is well-defined since modifying $J^{\eqv}$ by $J^{\eqv} a$, for $a \in (\Ue)^{\SL_2}(\ZZ^{\sv})$ changes $\chi$ by terms of lower length. This 
is because multiplication is trivial on the associated graded for the lower central series.
\end{proof}

Note  that the action of $\ssl_2$ on $\Or(\Ue)$ does not correspond to the action of the differential operators $\partial, \overline{\partial}$ 
 on $\MI^E$.

\begin{rem}  \label{rem2step}
In fact, since $\ue$  has no  $\SL_2$-invariant generators  it follows that $(\ue)^{\SL_2} \subset [\ue, \ue]$ and hence the map $\chi$ defined in the proof is well-defined (independent of the choice of $(b^{\sv},\phi^{\sv})$) on the two-step quotients $L^k/L^{k+2} \,   \lw( \Or( \Ue))$.
\end{rem}

\subsection{Orthogonality to cusp forms}
It follows from theorem \ref{thmpartialsorthogtocusp} that the composite 
$$\lw(\MI^E) \overset{\partial}{\To} \mathcal{M}_{\bullet, -1} \overset{p^h}{\To} S\ ,$$
where $S$ is the complex vector space  of holomorphic cusp forms,
is the zero map. This sheds some  light on the structure of 
  $\Or(\Ue)$, and makes it  clear that the generators of $\ue$ have infinitely many relations coming from every cusp form.

More precisely, consider the linear map
\begin{eqnarray}
 P\quad :  \quad  \big(E[\LL]  \otimes   \MI^E\big)_{\bullet,-1}   &\To &  S[\LL]   \nonumber \\
 \LL^{k+1}  \GE_{a } \otimes F_{b,k}  & \mapsto &      \LL^{k}  p^h ( \GE_{a }  F_{b,k}) \nonumber
\end{eqnarray}
where $F_{b,k} \in \MI^E$ has  modular weights $(b,k)$. Then 
theorem \ref{thmpartialsorthogtocusp} implies that 
$$\partial    \big( \lw( \MI^E ) \big)  \quad  \subset \quad \mathrm{ker} (P)$$
which  provides, via theorem \ref{thmMIEandOU}, a  possible constraint on $\mathrm{lw}(\Or(\Ue))$ for every  $f \LL^d \in S[\LL]$, where $f$ is an eigen cusp form and  $d \geq 0$.

\subsection{Pollack relations via orthogonality}

\begin{thm}   There is  an exact sequence $$0 \To \lw (\gr^L_2 \MI^E) \overset{\partial}{\To} \big(E[\LL] \otimes \gr^L_1 \MI^E\big)_{\bullet, -1} \overset{P}{\To} S[\LL]\ .$$
\end{thm}
This theorem is a consequence of results in \cite{MMV}. 
The last map is well-defined because $L_1 \MI^E \cong \ZZ^{\sv} \oplus \gr^L_1 \MI^E$ is split.  
The injectivity of the first map follows from proposition \ref{propModularKernel}.
Since $\gr^L_1 \MI^E $ is generated by real analytic Eisenstein series $\mathcal{E}_{r,s}$, the holomorphic projections
can be computed \cite{MMV}, \S9 using the Rankin-Selberg method, and give special values of $L$-functions of cusp forms.   One can show by this method
that the last map in the above sequence is surjective after tensoring with $\C$.

  Via theorem \ref{thmMIEandOU}  one can deduce the dimension of $\gr^L_2 \lw(\Or(\Ue))$, and indeed a description of the relations in its dual,  $\gr^L_2 \ue$.   This provides a modular interpretation of Pollack's quadratic relations \cite{Pollack}, \S\ref{sect: Relations}.   Note that multiple zeta values play no role in this calculation.

\subsection{Linearized double shuffle equations} \label{sect: PLS}
A connection  between $\ue$ and linearised double shuffle equations was described in \cite{Sigma}.
We briefly recall the statement.

The evaluation map (\ref{evx}), for $x= \av$, provides an embedding
$$\ue \  \hookrightarrow  \  \Lie(\av,\bv) \subset T^c(\av,\bv).$$
  The tensor coalgebra $T^c(\av,\bv)$   is graded by the degree in $\bv$. 
Consider the linear map 
\begin{eqnarray} 
\rho: \gr^{r}_{\bv}\, T^c(\av,\bv) &  \To &   \Q(x_1,\ldots, x_r)    \nonumber \\
\av^{i_0} \bv \av^{i_1} \bv \ldots \bv \av^{i_r} & \mapsto &    { x_1^{i_1} \ldots x_r^{i_r}  \over x_1(x_1-x_2) \ldots (x_{r-1}-x_{r})x_r  } \  \nonumber
\end{eqnarray} 
for $r\geq 1$.
The space $\pls^r$ was defined in \cite{Sigma} to be a graded  vector space of  homogeneous rational functions $f$  in the variables $x_1,\ldots, x_r$, for all $r\geq 2$
with  the property  $$x_1(x_1-x_2) \ldots (x_{r-1} -x_r)x_r f  \ \in  \  \Q[x_1,\ldots, x_r]$$
 which satisfy the linearised double shuffle equations.  For  $r=2$, they are
 \begin{eqnarray} 
  f(x_1,x_2) + f(x_2,x_1) & = & 0 \nonumber \\
  f(x_1,x_1+x_2) + f(x_2 ,x_1+x_2) & = & 0 \nonumber 
  \end{eqnarray}
  and for $r=3$, these equations take the form
 \begin{eqnarray} 
  f(x_1,x_2,x_3) + f(x_2,x_1,x_3) + f(x_2,x_3,x_1) & = & 0 \nonumber \\
  f(x_1,x_{12}, x_{123}) + f(x_2 ,x_{12},x_{123} )  + f(x_2 ,x_{23},x_{123} ) & = & 0 \nonumber 
  \end{eqnarray} 
  where $x_{ij} = x_i+x_j$ and $x_{123}= x_1+x_2+x_3$. 
  The space $\pls =\oplus_{r\geq 1} \pls^r$ is stable under a graded version of the Ihara Lie bracket, and is equipped with an  action of $\ssl_2$. 
\begin{thm} \label{thm: ueandpls} \cite{Sigma}
The linear map $\rho$, when restricted to $\ue$, defines  an injective map of bigraded Lie algebras
$$\rho: \ue \To \pls$$
which commutes with the action of $\ssl_2$ on both sides. 
\end{thm}

By theorem \ref{thmMIEandOU} we deduce a canonical surjection 
$$\gr^{\bullet}_L \lw ( \Or(\mathcal{P})) \otimes \ZZ^{\sv} \To   \gr^{\bullet}_L \lw (\MI^E)\ ,$$
where $\mathcal{P}$ denotes the affine group scheme corresponding to the graded Lie algebra $\pls$.
\\

The $\bv$-degree on $\ue$ coincides with the grading $r$ on $\pls$.  In \cite{Sigma}, we proved:

\begin{thm} The map  $\rho:\ue \rightarrow \pls$ is an isomorphism in $\bv$-degrees $\leq 3$. If 
$$u_k(s) = \sum_{n \in \Z} s^n  \dim_{\Q} (\gr^B_k  \gr^M_{2n} \ue) $$
denotes its Poincar\'e series, then 
$$ u_1 (s)  =   {s \over 1-s^2}  \quad  , \quad u_2 (s)   =   {s^2 \over (1-s^2)(1-s^6)} \quad , \quad u_3 (s)   =   {s \over (1-s^2)(1-s^4)(1-s^6)}$$
\end{thm}
Note that in \cite{Sigma}, we included $\varepsilon_0$ in the definition of $\ue$, which led to a marginally different formula for $u_1(s)$. 

\begin{rem} The bigraded Lie subalgebra 
$$\ls = \bigoplus_{r\geq1} \pls \cap \Q[x_1,\ldots, x_r]$$
of solutions to  linearised double shuffle equations in polynomials (i.e., with no poles) is  related to the structure of depth-graded motivic multiple zeta values \cite{IKZ, BrDepth}.  It follows from theorem \ref{thmMIEandOU} that the structure of $\MI^E$ is intimately connected to the structure of depth-graded multiple zeta values.  It would be interesting to compare with \cite{GKZ}. 
\end{rem} 

\section{$L$-functions associated to modular forms in $\MI^E$}
Hecke associated an $L$-function to every  classical holomorphic modular form. One can do the same for   functions in $\MI^E$. 
 
\begin{defn} Let $f\in \MI^E$ and let $f^0$ denote its constant part. Its completed $L$-function is defined for $\mathrm{Re}(s)$ large by  the regularised Mellin transform:
$$\Lambda(f; s)  = \int_0^{\infty} \left( f(iy) - f^0(iy)  \right) y^s \frac{dy}{y} \ . $$   
\end{defn} 
The definition applies to a much more general class of modular functions  (\cite{ZagFest}, \S10).

\begin{prop} If $f$ has modular weights $(\alpha, \beta)$, the function $\Lambda(f; s)$ admits a meromorphic continuation to $\C$ and satisfies the functional equation 
$$  \Lambda(f;  s) = i^{h}  \Lambda(f; w-s)  $$
where $w= \alpha+\beta$ and $h = \alpha- \beta$.   It has at most simple poles of the form
$$\sum_k  (-2\pi )^k   a_{0,0}^{(k)}  \left( \frac{i^h}{s-w-k} - \frac{1}{s+k} \right)\ , $$
where $f^0 = \sum_k  a^{(k)}_{0,0} \LL^k$ is the constant part of $f$. Furthermore, it can be expressed in terms of Dirichlet series as follows. Let
$$   L^{(k)}(f; s)  =\sum_{\substack{m,n \geq 0 \\    m+n \geq 1}}  \frac{a^{(k)}_{m,n} }{(m+n)^s}$$
where $a^{(k)}_{m,n}$ are the expansion coefficients  of \eqref{intro: fqcoeffs} of $f$. Then 
$$\Lambda(f; s) = \sum_k (-1)^k (2\pi)^{-s} \Gamma(s+k)  L^{(k)}(f;s+k) \ . $$
\end{prop}
 
Since the functions in $ \MI^E$ are naturally associated to universal mixed elliptic motives, their  $L$-functions are liable to contain 
interesting arithmetic information.  We do not know  the values of $\Lambda(f;s)$ for integers $s>w$, even when
 $f$ a product of  two real analytic Eisenstein series.

\subsection{Compatibility with the differential structure}
\begin{lem} \label{lemLambdapartial} The following relation holds for all $s\in \C$, where $w= \alpha+ \beta$ 
$$\Lambda(\partial f ; s) + \Lambda(\overline{\partial} f ; s) + (2s - w) \Lambda(f;s)=0\ .$$
\end{lem}
\begin{proof} Set $z= i y$, where $y \in \R$.  For $\mathrm{Re}(s)$ sufficiently large,  take the regularised Mellin transform from $z$ to  $\infty$ (see \cite{MultipleL} for the definition)  of both sides of  the following equation,
which follows from the definition of $\partial_{\alpha}, \overline{\partial}_{\beta}$: 
 $$d \left(f y^s \right) =  \left( y \frac{\partial f}{\partial y}  + s f \right) y^{s-1} dy  =\frac{1}{2} \left( \partial_{\alpha} f + \overline{\partial}_{\beta} f - w f \right) y^{s-1} dy + s   f  y^{s-1} dy  \  .$$
Conclude by analytic contination for all $s\in \C$. 
Alternatively, one can simply apply lemma 2.6 in \cite{ZagFest} which relates the expansion coefficients $a^{(k)}_{m,n}$ of 
$f, \partial f$ and $\overline{\partial} f$, and substitute into the formula for the Dirichlet series $L^{(k)}( \bullet, s)$.  \end{proof}

\begin{example}  Using lemma \ref{lemLambdapartial} and the differential equations \S\ref{sectRAEisenstein} for $\mathcal{E}_{\alpha,\beta}$ one obtains equations which one can solve to  express the $L$-functions of real analytic Eisenstein series in terms of Hecke $L$-functions of holomorphic Eisenstein series. This follows from the proof of  the theorem stated below.  More precisely,  one finds that:
$$\Lambda(\mathcal{E}_{\alpha, \beta}; s)  = \frac{\pi\, P_{\alpha,\beta}(s)}{s (s-2) \ldots (s-w-2) (s-w)} \,\Lambda(\G_{w+2}; s+1)$$ 
where $w= \alpha+\beta$, and $P_{\alpha,\beta}(s) \in \Q[s]$ has degree less than $w/2$ and  satisfies  $$P_{\alpha,\beta}(w-s) = P_{\alpha,\beta}(s)\ .$$  
 Since $\Lambda(\overline{f};s) = \Lambda(f;s)$ we have 
 $\Lambda(\mathcal{E}_{\alpha, \beta}; s) =\Lambda(\mathcal{E}_{\beta, \alpha}; s)$
and so $P_{\alpha, \beta}$ is symmetric in $\alpha, \beta$.  
 The function $\Lambda(\mathcal{E}_{\alpha, \beta}; s)$ has  simple poles at $s=-1,w+1$ with rational residues, and  at $0,w$
 with residues proportional to $\zeta(w+1)$.
 Some examples:
\begin{eqnarray}
\Lambda(\mathcal{E}_{2,0};s) \ = \  \Lambda(\mathcal{E}_{0,2};s) &=&  \frac{(s-1) \pi }{s(s-2)}  \Lambda(\G_4, s+1)  \nonumber \\
\Lambda(\mathcal{E}_{1,1};s) &=&  \frac{-2\pi }{s(s-2)}  \Lambda(\G_4, s+1)  \nonumber 
 \end{eqnarray} 
 where $\Lambda(\G_{2n};s) = (2\, \pi)^{-s}  \Gamma(s)  \zeta(s) \zeta(s-2n+1)$. 
 In general, the middle term $P_{k,k}\in \Q$ is constant. In fact, we claim that one obtains  $$ (- 4\pi)^k \, \Lambda(\mathcal{E}_{k,k};s) =  \frac{(2k-1)!}{(k-1)!} \, \xi(s+1) \xi(s -2 k) $$
where $\xi(s) = \pi^{-s/2} \Gamma(s/2)\zeta(s)$ is the completed Riemann zeta function. 
\end{example}

\subsection{Relation between non-holomorphic and mixed $L$-functions.}
In \cite{MultipleL} we  defined  $L$-functions in several variables  as iterated Mellin transforms of theta functions.

\begin{thm} \label{thm: relateLambdaNonHolandLambdaMultiple}  Let $F\in \MI^E_{\ell}$ have modular weights $\alpha,\beta \geq 0 $, with $w= \alpha+\beta$. Then 
$$s (s-1) \ldots (s-w) \, \Lambda(F;s)$$
is a $\mathcal{Z}^{\sv}[\pi]$-linear combination  of  the mixed $L$-functions  defined in \cite{MultipleL}:
$$ \Lambda(f_1,\ldots, f_k; a_1, \ldots, a_{k-1}, a_k+s)$$
where $k\leq \ell$, $f_1,\ldots, f_k$ are holomorphic Eisenstein series, and $a_1,\ldots, a_k$ are integers such that $1\leq a_i \leq \mathrm{weight}(f_i)-1$  for all $1\leq i \leq {k-1}$. 
\end{thm}
\begin{proof} By theorem \ref{thm: moddiff}, $\MI_k^E$ is generated by families of functions 
$F_{\alpha, \beta}$ satisfying
\begin{eqnarray}  
\partial F_{\alpha,\beta}  - (\alpha+1) F_{\alpha+1,\beta-1} & \in & E [\Lef] \times \MI^E_{k-1} \nonumber  \\ 
\overline{\partial} F_{\alpha,\beta}  - (s+1) F_{\alpha-1,\beta+1}  & \in & \overline{E}[\Lef] \times \MI^E_{k-1}    \nonumber 
\end{eqnarray} 
where we write $F_{-1,*} = F_{*,-1}=0$, and $E$ denotes the graded $\Q$-vector space generated by holomorphic Eisenstein series. 
Let $\alpha+\beta = w$ and write $\Lambda_{\alpha, \beta} = \Lambda(F_{\alpha,\beta};s)$ and $\Lambda_{-1,*}=\Lambda_{*,-1}=0$. Lemma \ref{lemLambdapartial} implies that
$$(\alpha+1) \Lambda_{\alpha+1,\beta-1}  +  (2s - w) \Lambda_{\alpha,\beta} +  (\beta+1) \Lambda_{\alpha-1,\beta+1}  \equiv 0 $$
modulo regularised Mellin transforms of elements in $(E+\overline{E})[\Lef] \times \MI^E_{k-1}$. We first need to check that this system of equations can be solved for the $\Lambda_{\alpha,\beta}$. This is obviously true for generic $s$ since the determinant of this  system of linear  equations
defines  a polynomial in $s$ and has finitely many zeros. In fact, one easily verifies that  the  matrix
$$M_w= \begin{pmatrix}
2s-w & 1 &     & \\
w & 2s-w & 2  &  \\
& w-1 & 2s-w & 3 &  \\
 &&  \ddots &  \ddots \\
& & & 2 & 2s-w & w \\
& & & & 1 & 2s-w 
\end{pmatrix}
$$
has determinant  
$\det (M_w) = 2^{w+1} \, s(s-1) \ldots (s-w) $. We deduce that 
$$s(s-1) \ldots (s-w)\, \Lambda(F_{\alpha, \beta}; s) $$
is a $\mathcal{Z}^{\sv}[\pi]$-linear combination of regularised Mellin transforms \cite{MultipleL}
\begin{equation} \label{inproofMToffA}  \int_0^{\!\tone_{\infty}} f(iy )   \mathcal{A}(iy )\,  y^{s+p} \frac{d y}{y}
\end{equation} 
and their complex conjugates, 
where $f \in E$, $\mathcal{A} \in \MI^E_{k-1}$, $p \in \N$.  By  definition of $J^{\eqv}$,  the restriction of $\mathcal{A}$ to the imaginary axis  can be expressed as a linear combination over $\mathcal{Z}^{\sv}[\pi]$ of regularised iterated integrals of length at most  $k-1$ of  holomorphic Eisenstein series $f_i(\tau) \tau^{r_i}$ where $0\leq r_i \leq n_i-2$ if $f_i$ is of weight $n_i$, since complex conjugation introduces at  most a sign.   Substituting into \eqref{inproofMToffA} gives  exactly the quantity $\Lambda(f_1,\ldots, f_{k-1}, f; r_1, \ldots, r_{k}, p+s)$.
\end{proof} 

This result holds for more general modular iterated integrals, but can also be made more precise in this situation: indeed,   the linear combination in the statement of the theorem is effectively computable.

\bibliographystyle{plain}
\bibliography{main}

\end{document}